\documentclass[12pt,twoside,reqno]{amsart}
\usepackage[latin1]{inputenc}
\usepackage[italian,english]{babel}
\usepackage{amssymb,latexsym}
\usepackage{amsmath,amsfonts,amsthm}
\usepackage{paralist}

\numberwithin{equation}{section}

\newcommand{\R}{\mathbb{R}}
\newcommand{\N}{\mathbb{N}}

\newcommand{\Z}{\mathbb{Z}}
\newcommand{\s}{\sharp}

\DeclareMathOperator{\dive}{div}

\newtheorem{lem}{Lemma}
\newtheorem{thm}{Theorem}

\newtheorem{defn}{Definition} 
\theoremstyle{remark}
\newtheorem{remark}{Remark}

\begin{document}

\title{periodic solutions for the non-local operator $(-\Delta+ m^{2})^{s}-m^{2s}$ with $m\geq 0$}

\author{Vincenzo Ambrosio}

\address
{\textsc {Vincenzo Ambrosio}\\
Dipartimento di Matematica e Applicazioni \\ 
Universit\`a degli Studi "Federico II" di Napoli \\
via Cinthia, 80126 Napoli, Italy }

\email{vincenzo.ambrosio2@unina.it}

\date{July 15, 2015}         


\begin{abstract}
By using variational methods we investigate the existence of $T$-periodic solutions to 
\begin{equation*}
\left\{
\begin{array}{ll}
[(-\Delta_{x}+m^{2})^{s}-m^{2s}]u=f(x,u) &\mbox{ in } (0,T)^{N}   \\
u(x+Te_{i})=u(x)    &\mbox{ for all } x \in \R^{N}, \quad i=1, \dots, N
\end{array}
\right.
\end{equation*}
where $s\in (0,1)$, $N>2s$, $T>0$, $m\geq 0$ and $f(x,u)$ is a continuous function, $T$-periodic in $x$, verifying the Ambrosetti-Rabinowitz condition  and a polynomial growth at rate $p\in (1, \frac{N+2s}{N-2s})$. 
\end{abstract}

\maketitle

\section{Introduction}

\noindent
Recently, considerable attention has been given to fractional Sobolev Spaces and corresponding nonlocal equations, in particular to the ones driven by the fractional powers of the Laplacian.
In fact, this operator naturally arises in several 
areas of research and find applications in optimization, finance, the thin obstacle problem, phase transitions, anomalous diffusion, crystal dislocation, flame propagation, conservation laws, ultra-relativistic limits of quantum mechanics, quasi-geostrophic flows and water waves. For more details and applications see \cite{Applebaum}, \cite{BKW}, \cite{Cabsolmor}, \cite{CV}, \cite{CT}, \cite{DL}, \cite{LY2}, \cite{Silvestre}, \cite{SirVal}, \cite{Stoker}, \cite{Toland} and references therein.

\noindent
The purpose of the present paper is to study the $T$-periodic solutions to the problem
\begin{equation}\label{P}
\left\{
\begin{array}{ll}
[(-\Delta_{x}+m^{2})^{s}-m^{2s}]u=f(x,u) &\mbox{ in } (0,T)^{N}   \\
u(x+Te_{i})=u(x)    &\mbox{ for all } x \in \R^{N}, \quad i=1, \dots, N
\end{array},
\right.
\end{equation}
where $s\in (0,1)$, $N >2s$, $(e_{i})$ is the canonical basis in $\R^{N}$ and $f:\R^{N+1}\rightarrow \R$ is a function satisfying the following hypotheses:
\begin{compactenum}[($f1$)]
\item $f(x,t)$ is $T$-periodic in $x \in \R^{N}$; that is $f(x+Te_{i},t)=f(x,t)$;
\item $f$ is continuous in $\R^{N+1}$; 
\item $f(x,t)=o(t)$ as $t \rightarrow 0$ uniformly in $x\in \R^{N}$;
\item there exist $1<p<2^{\s}_{s}-1=\frac{2N}{N-2s}-1$ and $C>0$ such that
$$
|f(x,t)|\leq C(1+|t|^{p})
$$
for any $x\in \R^{N}$ and $t\in \R$;
\item there exist $\mu>2$ and $r_{0}>0$ such that
$$
0<\mu F(x,t)\leq t f(x,t)
$$
for $x\in \R^{N}$ and $|t|\geq r_{0}$. Here $\displaystyle{F(x,t)=\int_{0}^{t} f(x,\tau) d\tau}$;
\item $tf(x,t) \geq 0$ for all $x\in \R^{N}$ and $t\in \R$.
\end{compactenum}
We notice that $(f2)$ and $(f5)$ imply the existence of two constants $a,b>0$ such that 
\begin{equation*}
F(x,t) \geq a|t|^{\mu}-b \mbox{ for all } x\in \R^{N}, t\in \R.
\end{equation*}
Then, since $\mu>2$, $F(x,t)$ grows at a superquadratic rate and by $(f5)$, $f(x,t)$ grows at a superlinear rate as $|t| \rightarrow \infty$. \\
Here, the operator $(-\Delta_{x}+m^{2})^{s}$ is defined through the spectral decomposition, by using the powers of the eigenvalues of $-\Delta+m^{2}$ with periodic boundary conditions.\\
Let $u\in\mathcal{C}^{\infty}_{T}(\R^{N})$, that is $u$ is infinitely differentiable in $\R^{N}$ and $T$-periodic in each variable.
Then $u$ has a Fourier series expansion:
$$
u(x)=\sum_{k\in \Z^{N}} c_{k} \frac{e^{\imath \omega k\cdot x}}{{\sqrt{T^{N}}}} \quad (x\in \R^{N})
$$
where 
$$
 \omega=\frac{2\pi}{T}\mbox{ and } \; c_{k}=\frac{1}{\sqrt{T^{N}}} \int_{(0,T)^{N}} u(x)e^{- \imath \omega k \cdot x}dx \quad (k\in \Z^{N})
$$
are the Fourier coefficients of $u$.
The operator $(-\Delta_{x}+m^{2})^{s}$ is defined by setting 
\begin{equation*}\label{nfrls}
(-\Delta_{x}+m^{2})^{s} \,u=\sum_{k\in \Z^{N}} c_{k} (\omega^{2}|k|^{2}+m^{2})^{s} \, \frac{e^{\imath \omega k\cdot x}}{{\sqrt{T^{N}}}}.
\end{equation*}
\noindent
For  $\displaystyle{u=\sum_{k\in \Z^{N}} c_{k} \frac{e^{\imath \omega k\cdot x}}{{\sqrt{T^{N}}}}}$ and $\displaystyle{v=\sum_{k\in \Z^{N}} d_{k} \frac{e^{\imath \omega k\cdot x}}{{\sqrt{T^{N}}}}}$, we have that 
$$
\mathcal{Q}(u,v)=\sum_{k\in \Z^{N}} (\omega^{2}|k|^{2}+m^{2})^{s} c_{k} \bar{d}_{k}
$$
can be extended by density to a quadratic form on the Hilbert space 
$$
\mathbb{H}^{s}_{m,T}=\Bigl\{u=\sum_{k\in \Z^{N}} c_{k} \frac{e^{\imath \omega k\cdot x}}{{\sqrt{T^{N}}}}\in L^{2}(0,T)^{N}: \sum_{k\in \Z^{N}} (\omega^{2}|k|^{2}+m^{2})^{s} \, |c_{k}|^{2}<\infty \Bigr\}
$$
endowed with the norm
$$
|u|_{\mathbb{H}^{s}_{m,T}}^{2}=\sum_{k\in \Z^{N}} (\omega^{2}|k|^{2}+m^{2})^{s} |c_{k}|^{2}.
$$
When $m=1$ we set $\mathbb{H}^{s}_{T}=\mathbb{H}^{s}_{1,T}$. \\
In $\R^{N}$, the physical interest of the non-local operator $(-\Delta+m^{2})^{s}$ is manifest in the case $s=1/2$: it is the Hamiltonian for a (free) relativistic particle of mass $m$; see for instance \cite{A1}, \cite{FL}, \cite{LL}, \cite{LY1}, \cite{LY2}. 
In particular, such operator is deeply connected with the Stochastic Process Theory: in fact it is an infinitesimal generator of a Levy process called $\alpha$-stable process; see \cite{Applebaum}, \cite{CMS} and \cite{Ryz}.\\
Problems similar to (\ref{P}) have been also studied in the local setting. 
The typical example is given by
\begin{equation}\label{P''}
\left\{
\begin{array}{ll}
Lu=f(x,u) &\mbox{ in } \Omega   \\
u=0   &\mbox{ on } \partial \Omega
\end{array},
\right.
\end{equation} 
where $L$ is uniformly elliptic, $\Omega$ is a smooth bounded domain in $\R^{N}$ and $f(x,t)$ is a continuous function satisfying the assumptions $(f3)$-$(f5)$. It is well known that (\ref{P''}) possesses a weak solution which can be obtained as critical point of a corresponding functional by means of minimax methods; see for instance \cite{AR}, \cite{MawWill}, \cite{Rab}, \cite{Struwe} and \cite{Willem}.
\smallskip

\noindent
The aim of the following paper is to study (\ref{P''}) in the periodic setting, when we replace
$L$ by $(-\Delta+m^{2})^{s}-m^{2s}$, $m\geq 0$ and $s\in (0,1)$.
We remark that the problem $(\ref{P})$ with $s=\frac{1}{2}$ has been investigated by the same author in \cite{A2}. In this paper, we extend the results in \cite{A2} to the more general operator $(-\Delta+m^{2})^{s}-m^{2s}$, with $s\in (0,1)$.
\medskip

\noindent
Our first result is the following 
\begin{thm}\label{thm1}
Let $m>0$ and $f:\R^{N+1} \rightarrow \R$ be a function satisfying the assumptions $(f1)-(f6)$.
Then there exists a solution $u\in \mathbb{H}^{s}_{m,T}$ to (\ref{P}). In particular, $u$ belongs to $\mathcal{C}^{0,\alpha}([0,T]^{N})$ for some $\alpha \in (0,1)$.
\end{thm}

\noindent
To study the problem $(\ref{P})$ we will give an alternative formulation of the operator $(-\Delta+m^{2})^{s}$ with periodic boundary condition which consists to realize it as an operator that maps a Dirichlet boundary condition to a Neumann-type condition via an extension problem on the half-cylinder $(0,T)^{N} \times (0,\infty)$; see \cite{A2} for the case $s= \frac{1}{2}$.
We recall that this argument is an adaptation of the idea originally introduced in \cite{CafSil} to study the fractional Laplacian in $\R^{N}$ (see also \cite{CabSir1, CabSir2}) and subsequently generalized for the case of the fractional Laplacian on bounded domain (see \cite{CabTan,  CapDavDupSir}). 

\noindent
As explained in more detail in Section $3$ below, for $u\in \mathbb{H}^{s}_{m,T}$ one considers the problem
\begin{equation*}
\left\{
\begin{array}{ll}
-\dive(y^{1-2s} \nabla v)+m^{2}y^{1-2s}v =0 &\mbox{ in }\mathcal{S}_{T}:=(0,T)^{N} \times (0,\infty)  \\
v_{| {\{x_{i}=0\}}}= v_{| {\{x_{i}=T\}}} & \mbox{ on } \partial_{L}\mathcal{S}_{T}:=\partial (0,T)^{N} \times [0,\infty) \\
v(x,0)=u(x)  &\mbox{ on }\partial^{0}\mathcal{S}_{T}:=(0,T)^{N} \times \{0\}
\end{array},
\right.
\end{equation*}
from where the operator $(-\Delta_{x}+m^{2})^{s}$ is obtained as 
$$
-\lim_{y\rightarrow 0} y^{1-2s} \frac{\partial v}{\partial y}(x,y) = \kappa_{s} (-\Delta_{x} + m^{2})^{s} u(x) 
$$
in weak sense and $\displaystyle{\kappa_{s}= 2^{1-2s} \frac{\Gamma(1-s)}{\Gamma(s)}}$.

\noindent
Thus, in order to study (\ref{P}), we will exploit this fact to investigate the following problem
\begin{equation}\label{R}
\left\{
\begin{array}{ll}
-\dive(y^{1-2s} \nabla v)+m^{2}y^{1-2s}v =0 &\mbox{ in }\mathcal{S}_{T}:=(0,T)^{N} \times (0,\infty)  \\
v_{| {\{x_{i}=0\}}}= v_{| {\{x_{i}=T\}}} & \mbox{ on } \partial_{L}\mathcal{S}_{T}:=\partial (0,T)^{N} \times [0,\infty) \\
\frac{\partial v}{\partial \nu^{1-2s}}=\kappa_{s} [m^{2s}v+f(x,v)]   &\mbox{ on }\partial^{0}\mathcal{S}_{T}:=(0,T)^{N} \times \{0\}
\end{array},
\right.
\end{equation}
where 
$$
\frac{\partial v}{\partial \nu^{1-2s}}:=-\lim_{y \rightarrow 0} y^{1-2s} \frac{\partial v}{\partial y}(x,y)
$$
is the conormal exterior derivative of $v$.\\
The solutions to $(\ref{R})$ are obtained as critical points of the functional $\mathcal{J}_{m}$ associated to (\ref{P})
 $$
 \mathcal{J}_{m}(v)=\frac{1}{2} ||v||_{\mathbb{X}_{m,T}^{s}}^{2}-\frac{m^{2s}\kappa_{s}}{2}|v(\cdot,0)|_{L^{2}(0,T)^{N}}^{2} -\kappa_{s}\int_{\partial^{0}\mathcal{S}_{T}} F(x,v) \,dx
 $$
defined on the space $\mathbb{X}_{m,T}^{s}$, which is the closure of the set of smooth and $T$-periodic (in $x$) functions in $\R^{N+1}_{+}$ with respect to the norm
$$
 ||v||_{\mathbb{X}_{m,T}^{s}}^{2}:=\iint_{\mathcal{S}_{T}} y^{1-2s} (|\nabla v|^{2}+m^{2s} v^{2}) dx dy.
$$ 
More precisely, we will prove that, for any fixed $m>0$, $\mathcal{J}_{m}$ satisfies the hypotheses of the Linking Theorem due to Rabinowitz \cite{Rab}.
\medskip

\noindent
When $m$ is sufficiently small, we are able to obtain estimates on critical levels $\alpha_{m}$ of the functionals $\mathcal{J}_{m}$ independently of $m$.  
In this way,  we can  pass to the limit as $m\rightarrow 0$ in (\ref{R}) and we deduce the existence of a nontrivial solution to the problem
\begin{equation}\label{P'}
\left\{
\begin{array}{ll}
(-\Delta_{x})^{s} u=f(x,u) &\mbox{ in } (0,T)^{N}   \\
u(x+Te_{i})=u(x)    &\mbox{ for all } x \in \R^{N}, \quad i=1, \dots, N
\end{array}.
\right.
\end{equation}
This result can be stated as follows
\begin{thm}\label{thmdue}
Under the same assumptions on $f$ of Theorem \ref{thm1}, 
the problem (\ref{P'}) admits a nontrivial solution $u\in \mathbb{H}^{s}_{T}\cap \mathcal{C}^{0,\alpha}([0,T]^{N})$.
\end{thm}

\noindent
The paper is organized as follows: in Section $2$ we collect some preliminaries results which we will use later to study the problem $(\ref{P})$; in Section $3$ we show that the problem (\ref{P}) can be realized in a local manner through the nonlinear problem (\ref{R}); 
in Section $4$ we verify that, for any fixed $m>0$, the functional $\mathcal{J}_{m}$ satisfies the Linking hypotheses; in Section $5$ we study the regularity of solutions of problem $(\ref{P})$; in the last Section we show that we can find a nontrivial H\"older continuous solution to $(\ref{P'})$ by passing to the limit in $(\ref{P})$ as $m\rightarrow 0$.

\section{Preliminaries}

In this section we introduce some notations and facts which will be frequently used in the sequel of paper.

We denote the upper half-space in $\R^{N+1}$ by 
$$
\R^{N+1}_{+}=\{(x,y)\in \R^{N+1}: x\in \R^{N}, y>0 \}.
$$
Let $\mathcal{S}_{T}=(0,T)^{N}\times(0,\infty)$ be the half-cylinder in $\R^{N+1}_{+}$ with basis $\partial^{0}\mathcal{S}_{T}=(0,T)^{N}\times \{0\}$
and we denote by $\partial_{L}\mathcal{S}_{T}=\partial (0,T)^{N}\times [0,+\infty)$ its lateral boundary.

With $||v||_{L^{r}(\mathcal{S}_{T})}$ we always denote the norm of $v\in L^{r}(\mathcal{S}_{T})$ and with $|u|_{L^{r}(0,T)^{N}}$ the $L^{r}(0,T)^{N}$ norm of $u \in L^{r}(0,T)^{N}$.

Let $s\in (0,1)$ and $m> 0$. Let $A\subset \R^{N}$ be a domain. We denote by $L^{2}(A\times \R_{+},y^{1-2s})$
the space of all measurable functions $v$ defined on $A\times \R_{+}$ such that
$$
\iint_{A\times \R_{+}} y^{1-2s} v^{2} dxdy<\infty.
$$
We say that  $v\in H^{1}_{m}(A\times \R_{+},y^{1-2s})$ if $v$ and its weak gradient $\nabla v$ belong to $L^{2}(A\times \R_{+},y^{1-2s})$.
The norm of $v$ in $H^{1}_{m}(A\times \R_{+},y^{1-2s})$ is given by
$$
\iint_{A\times \R_{+}} y^{1-2s} (|\nabla v|^{2}+m^{2}v^{2}) \,dxdy<\infty.
$$
It is clear that $H^{1}_{m}(A\times \R_{+},y^{1-2s})$ is a Hilbert space with the inner product
$$
\iint_{A\times \R_{+}} y^{1-2s} (\nabla v \nabla z+m^{2}v z) \,dxdy.
$$
When $m=1$, we set  $H^{1}(A\times \R_{+},y^{1-2s})=H^{1}_{1}(A\times \R_{+},y^{1-2s})$.

We denote by $\mathcal{C}^{\infty}_{T}(\R^{N})$ the space of functions 
$u\in \mathcal{C}^{\infty}(\R^{N})$ such that $u$ is $T$-periodic in each variable, that is
$$
u(x+e_{i}T)=u(x) \mbox{ for all } x\in \R^{N}, i=1, \dots, N. 
$$
Let $u\in  \mathcal{C}^{\infty}_{T}(\R^{N})$. Then we know that 
$$
u(x)=\sum_{k\in \Z^{N}} c_{k} \frac{e^{\imath\omega k\cdot x}}{\sqrt{T^{N}}} \quad \mbox{ for all } x\in \R^{N},
$$
where 
$$
\omega=\frac{2\pi}{T} \quad \mbox{ and } \quad c_{k}=\frac{1}{\sqrt{T^{N}}} \int_{(0,T)^{N}} u(x)e^{-\imath k\omega \cdot x}dx \quad (k\in \Z^{N})
$$ 
are the Fourier coefficients of $u$. We define the fractional Sobolev space $\mathbb{H}^{s}_{m,T}$ as the closure of 
$\mathcal{C}^{\infty}_{T}(\R^{N})$ under the norm 
\begin{equation*}\label{h12norm}
|u|^{2}_{\mathbb{H}^{s}_{m,T}}:= \sum_{k\in \Z^{N}} (\omega^{2}|k|^{2}+m^{2})^{s} \, |c_{k}|^{2}. 
\end{equation*}
When $m=1$, we set $\mathbb{H}^{s}_{T}=\mathbb{H}^{s}_{1,T}$ and $|\cdot |_{\mathbb{H}^{s}_{T}}=|\cdot|_{\mathbb{H}^{s}_{1,T}}$.
Now we introduce the functional space $\mathbb{X}^{s}_{m,T}$ defined as the completion of 
\begin{align*}
\mathcal{C}_{T}^{\infty}(\overline{\R^{N+1}_{+}})=\Bigl\{ & v\in \mathcal{C}^{\infty}(\overline{\R^{N+1}_{+}}): v(x+e_{i}T,y)=v(x,y) \\
&\mbox{ for every } (x,y)\in \overline{\R_{+}^{N+1}}, i=1, \dots, N \Bigr\}
\end{align*}
under the $H^{1}_{m}(\mathcal{S}_{T},y^{1-2s})$ norm 
\begin{equation*}
||v||^{2}_{\mathbb{X}^{s}_{m,T}}:=\iint_{\mathcal{S}_{T}} y^{1-2s} (|\nabla v|^{2}+m^{2}v^{2}) \, dxdy.
\end{equation*} 
If $m=1$, we set $\mathbb{X}^{s}_{T}=\mathbb{X}^{s}_{1,T}$ and $||\cdot ||_{\mathbb{X}^{s}_{T}}=||\cdot||_{\mathbb{X}^{s}_{1,T}}$. 

We begin proving that it is possible to define a trace operator from $\mathbb{X}^{s}_{m,T}$ to the fractional space $\mathbb{H}^{s}_{m,T}$:

\begin{thm}\label{tracethm}
There exists a bounded linear operator $\textup{Tr} : \mathbb{X}^{s}_{m,T} \rightarrow \mathbb{H}^{s}_{m,T}$  such that :
\begin{itemize}
\item[(i)] $\textup{Tr}(v)=v|_{\partial^{0} \mathcal{S}_{T}}$ for all $v\in \mathcal{C}_{T}^{\infty}(\overline{\R^{N+1}_{+}}) \cap \mathbb{X}^{s}_{m,T}$;
\item[(ii)] There exists $C=C(s)>0$ such that  
$$C |\textup{Tr}(v)|_{\mathbb{H}^{s}_{m,T}}\leq ||v||_{\mathbb{X}^{s}_{m,T}} \mbox{ for every } v\in \mathbb{X}^{s}_{m,T}; 
$$ 
\item[(iii)] $\textup{Tr}$ is surjective.
\end{itemize}
\end{thm} 
\begin{proof}
Let $v\in C^{\infty}_{T}(\overline{\R^{N+1}_{+}})$ such that $\|v\|_{\mathbb{X}^{s}_{m,T}}<\infty$. Then $v$ can be expressed by
$$
v(x,y)= \sum_{k\in \Z^{N}} c_{k}(y) \frac{e^{\imath \omega k\cdot x}}{{\sqrt{T^{N}}}}
$$
where $c_{k}(y)= \int_{(0,T)^{N}} v(x,y) \frac{e^{-\imath \omega k\cdot x}}{{\sqrt{T^{N}}}} \, dx$ and $c_{k}\in H^{1}_{m}(\R_{+}, y^{1-2s})$. 

We notice that, by using Parseval's identity, we have 
\begin{equation}\label{A0}
\|v\|^{2}_{\mathbb{X}^{s}_{m,T}} = \sum_{k\in \Z^{N}} \int_{0}^{+\infty} y^{1-2s} \, [ (\omega^{2} |k|^{2}+m^{2})  |c_{k}(y)|^{2} +  |c'_{k}(y)|^{2}] \, dy. 
\end{equation}

Then, we are going to prove that there exists a positive constant $C_{s}$ depending only on $s$, such that
$$
C_{s} |\textup{Tr}(v)|_{\mathbb{H}^{s}_{m,T}}^{2} \leq \|v\|_{\mathbb{X}^{s}_{m,T}}^{2} \, \mbox{ for any } v\in C^{\infty}_{T}(\overline{\R^{N+1}_{+}}) : \|v\|_{\mathbb{X}^{s}_{m,T}}<+\infty, 
$$
or equivalently
\begin{align}\label{A33}
C_{s}& \sum_{k\in \Z^{N}} (\omega^{2} |k|^{2}+m^{2})^{s} |c_{k}(0)|^{2} \nonumber \\
&\leq  \sum_{k\in \Z^{N}} \int_{0}^{+\infty} y^{1-2s} \, [ (\omega^{2} |k|^{2}+m^{2})  |c_{k}(y)|^{2} +  |c'_{k}(y)|^{2}] \, dy.
\end{align}
By the Fundamental Theorem of Calculus we can see
\begin{align*}
|c_{k}(0)| \leq |c_{k}(y)| + \Bigl|\int_{0}^{y} c'_{k}(t) \, dt \Bigr| 
\end{align*}   
and by using $(|a|+|b|)^{2}\leq 2(|a|^{2}+|b|^{2})$ we have
\begin{align}\label{A11}
|c_{k}(0)|^{2} \leq 2|c_{k}(y)|^{2} + 2 \Bigl|\int_{0}^{y} c'_{k}(t) \, dt \Bigr|^{2}
\end{align}   
for any $k\in \Z^{N}$.
Now, we observe that by using H\"older inequality we obtain 
\begin{align}\label{A22}
\int_{0}^{y} |c'_{k}(t)|\, dt &\leq \Bigl(\int_{0}^{y} t^{1-2s} |c'_{k}(t)|^{2} dt\Bigr)^{\frac{1}{2}} \Bigl(\int_{0}^{y} t^{2s-1} dt\Bigr)^{\frac{1}{2}} \nonumber \\
&= \Bigl(\int_{0}^{y} t^{1-2s} |c'_{k}(t)|^{2} dt\Bigr)^{\frac{1}{2}} \Bigl(\frac{y^{2s}}{2s}\Bigr)^{\frac{1}{2}}.
\end{align}
Then, putting together (\ref{A11}) and (\ref{A22}) we obtain 
\begin{align*}
|c_{k}(0)|^{2} \leq 2|c_{k}(y)|^{2} + \frac{y^{2s}}{s} \Bigl(\int_{0}^{+\infty} t^{1-2s} |c'_{k}(t)|^{2} dt\Bigr)
\end{align*}   
and multiplying both sides by $y^{1-2s}$ we get
\begin{align}\label{A3}
y^{1-2s} |c_{k}(0)|^{2} \leq 2 y^{1-2s} |c_{k}(y)|^{2} + \frac{y}{s} \Bigl(\int_{0}^{+\infty} t^{1-2s} |c'_{k}(t)|^{2} dt\Bigr).
\end{align}
Let $a_{k}= (\omega^{2} |k|^{2}+m^{2})^{-\frac{1}{2}}$. Integrating (\ref{A3}) over $y\in (0, a_{k})$ we deduce  
\begin{align}\label{A4}
\frac{a_{k}^{2-2s}}{2-2s} |c_{k}(0)|^{2} &\leq 2 \int_{0}^{a_{k}} y^{1-2s} |c_{k}(y)|^{2} dy + \Bigl(\int_{0}^{a_{k}} \frac{y}{s} \, dy \Bigr) \Bigl(\int_{0}^{+\infty} t^{1-2s} |c'_{k}(t)|^{2} dt\Bigr) \nonumber \\
&\leq 2 \int_{0}^{+\infty} y^{1-2s} |c_{k}(y)|^{2} dy + \frac{a_{k}^{2}}{2s} \Bigl(\int_{0}^{+\infty} t^{1-2s} |c'_{k}(t)|^{2} dt\Bigr) \nonumber \\
&= 2 \int_{0}^{+\infty} t^{1-2s} |c_{k}(t)|^{2} dt + \frac{a_{k}^{2}}{2s} \Bigl(\int_{0}^{+\infty} t^{1-2s} |c'_{k}(t)|^{2} dt\Bigr). 
\end{align}
Multiplying both sides of (\ref{A4}) by $a_{k}^{-2}= (\omega^{2} |k|^{2}+m^{2})$ we have 
\begin{align*}
\frac{(\omega^{2} |k|^{2}+m^{2})^{s}}{2-2s} |c_{k}(0)|^{2} &\leq 2 (\omega^{2} |k|^{2}+m^{2}) \int_{0}^{+\infty} t^{1-2s} |c_{k}(t)|^{2} dt \\
&+ \frac{1}{2s} \Bigl(\int_{0}^{+\infty} t^{1-2s} |c'_{k}(t)|^{2} dt\Bigr)
\end{align*}
for any $k\in \Z^{N}$.

Summing over $\Z^{N}$ we deduce
\begin{align}\label{A5}
\frac{1}{2-2s}  &\sum_{k\in \Z^{N}} (\omega^{2} |k|^{2}+m^{2})^{s} |c_{k}(0)|^{2} \leq \nonumber \\
&\leq  \sum_{k\in \Z^{N}} \Bigl[2(\omega^{2} |k|^{2}+m^{2}) \int_{0}^{+\infty} t^{1-2s} |c_{k}(t)|^{2} dt 
+ \frac{1}{2s} \Bigl(\int_{0}^{+\infty} t^{1-2s} |c'_{k}(t)|^{2} dt\Bigr)\Bigr] \nonumber \\
&\leq \max \Bigl\{2, \frac{1}{2s}\Bigr \} \sum_{k\in \Z^{N}} \int_{0}^{+\infty} t^{1-2s} \, [ (\omega^{2} |k|^{2}+m^{2})  |c_{k}(t)|^{2} +  |c'_{k}(t)|^{2}] \, dt.  
\end{align}
Taking into account (\ref{A0}) and (\ref{A5}) we have $(\ref{A33})$.

Therefore there exists a trace operator $\textup{Tr} : \mathbb{X}^{s}_{m,T} \rightarrow \mathbb{H}^{s}_{m,T}$. 
Now we prove that $\textup{Tr}$ is surjective. Let $u= \sum_{k\in \Z^{N}} c_{k} \frac{e^{\imath \omega k\cdot x}}{{\sqrt{T^{N}}}} \in \mathbb{H}^{s}_{m,T}$. 

We define 
\begin{align}\label{extu}
v(x,y)=\sum_{k\in \Z^{N}} c_{k} \theta_{k}(y) \frac{e^{\imath \omega k\cdot x}}{{\sqrt{T^{N}}}}
\end{align}  
where $\theta_{k}(y)=  \theta(\sqrt{\omega^{2} |k|^{2}+ m^{2}} y)$ and $\theta(y)\in H^{1}(\R_{+},y^{1-2s})$ solves the following ODE
\begin{equation*} \label{ccv}
\left\{
\begin{array}{cc}
&\theta^{''}+\frac{1-2s}{y}\theta^{'}-\theta=0 \mbox{ in } \R_{+}  \\
&\theta(0)=1 \mbox{ and } \theta(\infty)=0
\end{array}.
\right.
\end{equation*}
It is known (see \cite{Erd}) that $\theta(y)=\frac{2}{\Gamma(s)} (\frac{y}{2})^{s} K_{s}(y)$ where $K_{s}$ is the Bessel function of second kind with order $s$, and
being $K'_{s}=\frac{s}{y}K_{s}-K_{s-1}$, we can see that
$$
\kappa_{s}:= \int_{0}^{\infty} y^{1-2s} (|\theta'(y)|^{2}+|\theta(y)|^{2}) dy=-\lim_{y \rightarrow 0} y^{1-2s}\theta'(y)=2^{1-2s} \frac{\Gamma(1-s)}{\Gamma(s)}.
$$
Then it is clear that $v$ is smooth for $y>0$, $v$ is $T$-periodic in $x$ and satisfies 
$$
-\dive(y^{1-2s} \nabla v)+m^{2}y^{1-2s}v =0 \mbox{ in } \mathcal{S}_{T} .
$$
Now, we show that $\textup{Tr}(v)=u$. From standard properties of $K_{s}$, we know $\theta(y) \rightarrow 1$ as $y \rightarrow 0$ and 
$0<\theta(y)\leq A_{s}$ for any $y \geq 0$.
Then, being $u\in \mathbb{H}^{s}_{m,T}$, we have
$$
|v(\cdot,y)-u|_{\mathbb{H}^{s}_{m,T}}^{2}=\sum_{k\in \Z^{N}}(\omega^{2}|k|^{2}+m^{2})^{s} |c_{k}|^{2} |\theta_{k}(y)-1|^{2} \rightarrow 0 \mbox{ as } y\rightarrow 0.
$$
Finally we prove that $v\in \mathbb{X}^{s}_{m,T}$. By Parseval's identity we get 
\begin{align}
||v||^{2}_{\mathbb{X}^{s}_{m,T}}&=\iint_{\mathcal{S}_{T}} y^{1-2s} (|\nabla v|^{2} + m^{2} v^{2} ) \, dxdy \nonumber \\
&=\sum_{k\in \Z^{N}} |c_{k}|^{2} \int_{0}^{\infty} y^{1-2s} (|\theta'_{k}(y)|^{2} +|\theta_{k}(y)|^{2} ) dy  \nonumber \\
&=\sum_{k\in \Z^{N}} |c_{k}|^{2} \int_{0}^{\infty} y^{1-2s} (\omega^{2} |k|^{2}+m^{2}) (|\theta'(\sqrt{\omega^{2} |k|^{2}+m^{2}}y)|^{2} +|\theta(\sqrt{\omega^{2} |k|^{2}+m^{2}}y)|^{2} ) dy  \nonumber \\
&=\sum_{k\in \Z^{N}} |c_{k}|^{2} \frac{\sqrt{\omega^{2} |k|^{2}+m^{2}}}{(\omega^{2} |k|^{2}+m^{2})^{\frac{1-2s}{2}}} \int_{0}^{\infty} y^{1-2s} (|\theta'(y)|^{2} +|\theta(y)|^{2} )dy  \nonumber \\
&=\kappa_{s} \sum_{k\in \Z^{N}} (\omega^{2} |k|^{2}+m^{2})^{s} |c_{k}|^{2}  \nonumber \\
&=\kappa_{s} |u|^{2}_{\mathbb{H}^{s}_{m,T}}. \label{v=u}
\end{align}

\end{proof}

Now we prove the following embedding
\begin{thm}\label{thm2}
Let $N> 2s$. Then  $\textup{Tr}(\mathbb{X}^{s}_{m,T})$ is continuously embedded in $L^{q}(0,T)^{N}$ for any  $1\leq q \leq 2^{\s}_{s}$.  Moreover,  $\textup{Tr}(\mathbb{X}^{s}_{m,T})$ is compactly embedded in $L^{q}(0,T)^{N}$  for any  $1\leq q < 2^{\s}_{s}$. 
\end{thm}
\begin{proof}
By Theorem \ref{tracethm} we know that there exists a continuous embedding from $\mathbb{X}^{s}_{m,T}$ to $\mathbb{H}^{s}_{m,T}$.
Now, we prove that $ \mathbb{H}^{s}_{m,T}$ is continuously embedded in $L^{q}(0,T)^{N}$ for any $q \leq 2^{\s}_{s}$ and compactly in $L^{q}(0,T)^{N}$ for any $q< 2^{\s}_{s}$.

By using Proposition $2.1.$ in \cite{benyi}, we know that there exists a constant $C_{2^{\s}_{s}}>0$ such that
\begin{align}\label{benyi}
|u|_{L^{2^{\s}_{s}}(0,T)^{N}} \leq C_{2^{\s}_{s}} \Bigl( \sum_{|k|\geq 1} \omega^{2s}|k|^{2s} |c_{k}|^{2} \Bigr)^{\frac{1}{2}} 
\end{align}
for any $u\in \mathcal{C}^{\infty}_{T}(\R^{N})$ such that $\frac{1}{T^{N}}\int_{(0,T)^{N}} u(x) \, dx =0$.
As a consequence, fixed $2 \leq q\leq 2^{\s}_{s}$, we have
\begin{equation}\label{6.22}
|u|_{L^{q}(0,T)^{N}}\leq C \Bigl( \sum_{k\in \Z^{N}} |c_{k}|^{2} (\omega^{2}|k|^{2}+m^{2})^{s}   \Bigr)^{\frac{1}{2}}
\end{equation}
for any $u\in \mathbb{H}^{s}_{m,T}$, that is $\mathbb{H}^{s}_{m,T}$ is continuously embedded in $L^{q}(0,T)^{N}$ for any $2 \leq q\leq 2^{\s}_{s}$. 
Now, we proceed as the proof of Theorem $4$ in \cite{A2} to prove that $\mathbb{H}^{s}_{m,T} \Subset L^{q}(0,T)^{N}$ for any $2 \leq q<2^{\s}_{s}$.
Fix $q\in [2,2^{\s}_{s})$. Then, by using  (\ref{6.22}) and the interpolation inequality we obtain 
\begin{equation}\label{6.23}
|u|_{L^{q}(0,T)^{N}}\leq C|u|^{\theta}_{L^{2}(0,T)^{N}} \Bigl( \sum_{k\in \Z^{N}} |c_{k}|^{2} (\omega^{2}|k|^{2}+m^{2})^{s}   \Bigr)^{1-\theta},
\end{equation}
for some real positive number $\theta\in (0,1)$.

Now, taking into account (\ref{6.23}), it is enough to prove that $\mathbb{H}^{s}_{m,T}\Subset L^{2}(0,T)^{N}$ to infer that $\mathbb{H}^{s}_{m,T}$ is compactly embedded in $L^{q}(0,T)^{N}$ for every $q\in [2,2^{\s}_{s})$. 

Let us assume that $u^{j}\rightharpoonup 0$ in $\mathbb{H}^{s}_{m,T}$. Then 
\begin{equation}\label{6.24}
\lim_{j\rightarrow \infty} |c^{j}_{k}|^{2}(\omega^{2}|k|^{2}+m^{2})^{s}=0 \quad \forall k\in \Z^{N} 
\end{equation}
and
\begin{equation}\label{6.25}
 \sum_{k\in \Z^{N}}  |c^{j}_{k}|^{2} (\omega^{2}|k|^{2}+m^{2})^{s}\leq C \quad \forall j\in \N.
\end{equation}
Fix $\varepsilon>0$. Then there exists $\nu_{\varepsilon}>0$ such that $(\omega^{2}|k|^{2}+m^{2})^{-s}<\varepsilon$ for $|k|>\nu_{\varepsilon}$. 
By (\ref{6.25}) we have
\begin{align*}
\sum_{k\in \Z^{N}} |c_{k}^{j}|^{2}&=\sum_{|k|\leq \nu_{\varepsilon}} |c_{k}^{j}|^{2}+\sum_{|k|>\nu_{\varepsilon}} |c_{k}^{j}|^{2} \\
&= \sum_{|k|\leq \nu} |c_{k}^{j}|^{2}+\sum_{|k|>\nu_{\varepsilon}} |c_{k}^{j}|^{2}(\omega^{2}|k|^{2}+m^{2})^{s} (\omega^{2}|k|^{2}+m^{2})^{-s} \\
&\leq \sum_{|k|\leq \nu_{\varepsilon}} |c_{k}^{j}|^{2}+C\varepsilon.
\end{align*}
Using (\ref{6.24}) we deduce that $\sum_{|k|\leq \nu_{\varepsilon}} |c_{k}^{j}|^{2}<\varepsilon$ for $j$ large. Thus $u^{j}\rightarrow 0$ in $L^{2}(0,T)^{N}$.

\end{proof}


Now, we conclude this section giving some elementary results about the nonlinearity $f(x,t)$. More precisely, by using the assumptions $(f2), (f3)$ and $(f4)$, we can deduce some bounds from above and below for the nonlinear term $f(x,t)$ and its primitive $F(x,t)$. This part is quite standard and the proofs of the following Lemmas can be found, for instance, in \cite{AR} and \cite{Rab}.
\begin{lem}\label{lfF}
Let $f:[0,T]^{N}\times \R \rightarrow \R$ satisfying conditions $(f1)$-$(f3)$. Then, for any $\varepsilon>0$ there 
exists $C_{\varepsilon}>0$ such that
\begin{equation*}\label{f}
|f(x,t)|\leq  2\varepsilon|t|+(p+1)C_{\varepsilon}|t|^{p} \quad \forall t\in{\R} \, \forall x\in [0,T]^{N}
\end{equation*}
and
\begin{equation}\label{F}
|F(x,t)|\leq  \varepsilon|t|^{2}+C_{\varepsilon}|t|^{p+1} \quad \forall t\in {\R} \, \forall x\in [0,T]^{N}.
\end{equation}
\end{lem}
\begin{lem}\label{F**}
Assume that $f:[0,T]^{N}\times \R \rightarrow \R$ satisfies conditions $(f1)$-$(f4)$. Then, there exist two constants $a_{3}>0$ and $a_{4}>0$  such that  
\begin{equation*}\label{F*}
F(x,t)\geq a_{3}|t|^{\mu}-a_{4} \quad \forall t\in{\R} \, \forall x\in [0, T]^{N}.
\end{equation*}
\end{lem}

\section{Extension problem}

In this section we show that to study $(\ref{P})$ it is equivalent to investigate the solutions of a problem in a half-cylinder with a Neumann nonlinear boundary condition.
We start proving that

\begin{thm}
Let $u\in \mathbb{H}^{s}_{m,T}$. Then there exists a unique $v\in \mathbb{X}^{s}_{m,T}$ such that
\begin{equation}\label{extPu}
\left\{
\begin{array}{ll}
-\dive(y^{1-2s} \nabla v)+m^{2}y^{1-2s}v =0 &\mbox{ in }\mathcal{S}_{T}  \\
v_{| {\{x_{i}=0\}}}= v_{| {\{x_{i}=T\}}} & \mbox{ on } \partial_{L}\mathcal{S}_{T} \\
v(\cdot,0)=u  &\mbox{ on } \partial^{0}\mathcal{S}_{T}
\end{array}
\right.
\end{equation}
and
\begin{align}\label{conormal}
-\lim_{y \rightarrow 0} y^{1-2s}\frac{\partial v}{\partial y}(x,y)=\kappa_{s} (-\Delta_{x}+m^{2})^{s}u(x) \mbox{ in } \mathbb{H}^{-s}_{m,T},
\end{align}
where 
$$
\mathbb{H}^{-s}_{m,T}=\Bigl \{u=\sum_{k\in \Z^{N}} c_{k} \frac{e^{\imath \omega k\cdot x}}{{\sqrt{T^{N}}}}: \frac{|c_{k}|^{2}}{(\omega^{2}|k|^{2}+m^{2})^{s}}< \infty \Bigr\}
$$ 
is the dual of $\mathbb{H}^{s}_{m,T}$.
\end{thm}
\begin{proof}
Let $u=\sum_{k\in \Z^{N}} c_{k} \frac{e^{\imath \omega k\cdot x}}{{\sqrt{T^{N}}}}\in \mathcal{C}^{\infty}_{T}(\R^{N})$.
Let us consider the following problem
\begin{equation}\label{minpb}
\min\{ ||v||^{2}_{\mathbb{X}^{s}_{m,T}}: v\in \mathbb{X}^{s}_{m,T}, \textup{Tr}(v)=u \}.
\end{equation}
By using the Theorem \ref{thm2}, we can find a minimizer to $(\ref{minpb})$. 
Since $||\cdot||^{2}_{\mathbb{X}^{s}_{m,T}}$ is strictly convex, such minimizer is unique and we denote it by $v$.
As a consequence, for any $\phi \in \mathbb{X}^{s}_{m,T}$ such that $\textup{Tr}(\phi)=0$
\begin{equation}\label{WE}
\iint_{\mathcal{S}_{T}} y^{1-2s} (\nabla v \nabla \phi + m^{2} v \phi  ) \, dxdy=0, 
\end{equation}
that is $v$ is a weak solution to $(\ref{extPu})$.
Since the function defined in (\ref{extu}) is a solution to (\ref{extPu}), by the uniqueness of minimizer, we deduce that $v$ is given by 
$$
v(x,y)=\sum_{k\in \Z^{N}} c_{k} \theta_{k}(y) \frac{e^{\imath \omega k\cdot x}}{{\sqrt{T^{N}}}}.
$$
where $\theta_{k}(y)=\theta(\sqrt{\omega^{2}|k|^{2}+ m^{2}} y)$. In particular, by (\ref{v=u}), we have 
$$
||v||_{\mathbb{X}^{s}_{m,T}}=\sqrt{\kappa_{s}} |u|_{\mathbb{H}^{s}_{m,T}}.
$$

Then
\begin{align*}
&\Bigl| -y^{1-2s} \frac{\partial v}{\partial y}(\cdot,y)-\kappa_{s} (-\Delta+m^{2})^{s}u \Bigr|^{2}_{\mathbb{H}^{-s}_{m,T}} \\
&=\sum_{k\in \Z^{N}} \frac{1}{(\omega^{2}|k|^{2}+m^{2})^{s}} |c_{k}|^{2} \Bigl|\sqrt{\omega^{2}|k|^{2}+m^{2}} \theta_{k}'(y)y^{1-2s}+\kappa_{s}(\omega^{2}|k|^{2}+m^{2})^{s} \Bigr|^{2} \\
&=\sum_{k\in \Z^{N}} (\omega^{2}|k|^{2}+m^{2})^{s}  |c_{k}|^{2} \Bigl|[(\omega^{2}|k|^{2}+m^{2})y]^{1-2s} \theta_{k}'(y)+\kappa_{s} \Bigr|^{2}
\end{align*}
and by using $u\in \mathbb{H}^{s}_{m,T}$, $-y^{1-2s} \theta'(y) \rightarrow \kappa_{s}$ as $y\rightarrow 0$ and $0<-\kappa_{s}y^{1-2s} \theta'(y)\leq B_{s}$ for any $y \geq 0$ (see \cite{Erd}), we deduce (\ref{conormal}).

\end{proof}

Therefore, for any given $u\in \mathbb{H}^{s}_{m,T}$ we can find a unique function $v=\textup{Ext}(u)\in \mathbb{X}^{s}_{m,T}$, which will be called the extension of $u$, such that
\begin{compactenum}[(i)]
\item $v$ is smooth for $y>0$, $T$-periodic in $x$ and $v$ solves (\ref{extPu});
\item $||v||_{\mathbb{X}^{s}_{m,T}}\leq ||z||_{\mathbb{X}^{s}_{m,T}}$ for any $z\in \mathbb{X}^{s}_{m,T}$ such that $\textup{Tr}(z)=u$;
\item $||v||_{\mathbb{X}^{s}_{m,T}}=\sqrt{\kappa_{s}} |u|_{\mathbb{H}^{s}_{m,T}}$;
\item We have
$$
\lim_{y\rightarrow 0}-y^{1-2s} \frac{\partial v}{\partial y}(x,y)= \kappa_{s} (-\Delta+m^{2})^{s}u(x) \mbox{ in } \mathbb{H}^{-s}_{m,T}.
$$
\end{compactenum}

Now, modifying the proof of Lemma $2.2$ in \cite{CapDavDupSir} we deduce 
\begin{thm}
Let $g\in \mathbb{H}^{-s}_{m,T}$. Then, there is a unique solution to the problem:
$$
\mbox{ find } u\in \mathbb{H}^{s}_{m,T} \quad \mbox{ such that } \quad (-\Delta+m^{2})^{s} u=g. 
$$
Moreover $u$ is the trace of $v\in \mathbb{X}^{s}_{m,T}$, where $v$ is the unique solution to (\ref{R}), that is
for every $\phi \in \mathbb{X}^{s}_{m,T}$ it holds
$$
\iint_{\mathcal{S}_{T}} y^{1-2s} (\nabla v \nabla \phi + m^{2} v \phi  ) \, dxdy=\kappa_{s}\langle g, \textup{Tr}(\phi)\rangle_{\mathbb{H}^{-s}_{m,T}, \mathbb{H}^{s}_{m,T}}. 
$$
\end{thm}

Taking into account the previous results we can reformulate the nonlocal problem (\ref{P}) in a local way as explained below.

Let $g \in \mathbb{H}^{-s}_{m,T}$ and consider the following two problems:
\begin{equation}\label{P*}
\left\{
\begin{array}{ll}
(-\Delta_{x}+m^{2})^{s}u=g &\mbox{ in } (0,T)^{N}  \\
u(x+Te_{i})=u(x) & \mbox{ for } x\in \R^{N}
\end{array}
\right.
\end{equation}
and
\begin{equation}\label{P**}
\left\{
\begin{array}{ll}
-\dive(y^{1-2s} \nabla v)+m^{2}y^{1-2s}v =0 &\mbox{ in }\mathcal{S}_{T}  \\
v_{| {\{x_{i}=0\}}}= v_{| {\{x_{i}=T\}}} & \mbox{ on } \partial_{L}\mathcal{S}_{T} \\
\frac{\partial v}{\partial \nu^{1-2s}}=g(x)  &\mbox{ on } \partial^{0}\mathcal{S}_{T}
\end{array}.
\right.
\end{equation}

\begin{defn}
We say that $u\in \mathbb{H}^{s}_{m,T}$ is a weak solution to (\ref{P*}) if $u=\textup{Tr}(v)$
and $v$ is a weak solution to (\ref{P**}).
\end{defn}

\begin{remark}
Later, with abuse of notation, we will denote by $v(\cdot,0)$ the trace $\textup{Tr}(v)$ of a function $v\in \mathbb{X}^{s}_{m,T}$.
\end{remark}

We conclude this section giving the proof of the following sharp trace inequality:
\begin{thm}\label{thm6}
For any $v\in \mathbb{X}^{s}_{m,T}$ we have 
\begin{equation}\label{eqYY}
 \kappa_{s} |\textup{Tr}(v)|_{\mathbb{H}^{s}_{m,T}}^{2}\leq ||v||_{\mathbb{X}^{s}_{m,T}}^{2}
\end{equation}
and the equality is attained if and only if $v=\textup{Ext}(\textup{Tr}(v))$.
In particular
\begin{equation}\label{eqY}
||v||_{\mathbb{X}^{s}_{m,T}}^{2} - \kappa_{s} m^{2s} |\textup{Tr}(v)|_{L^{2}(0,T)^{N}}^{2} =0 \Leftrightarrow v(x,y)=C\, \theta(my) \mbox{ for some } C\in \R. 
\end{equation}
\end{thm}
\begin{proof}
By properties $(ii)$ and $(iii)$ we can see that for any $v\in \mathbb{X}^{s}_{m,T}$
\begin{equation*}
 \kappa_{s} |\textup{Tr}(v)|_{\mathbb{H}^{s}_{m,T}}^{2}=||\textup{Ext}\textup{Tr}(v)||_{\mathbb{X}^{s}_{m,T}}^{2}\leq ||v||_{\mathbb{X}^{s}_{m,T}}^{2},
\end{equation*}
and the equality holds if and only if $v=\textup{Ext}\textup{Tr}(v)$.

Now, we prove (\ref{eqY}).
We denote by $c_{k}$ the Fourier coefficients of $\textup{Tr}(v)$.

If $v(x,y)= C \theta_{m}(y) : = C \theta(my)$ for some $C\in \R$, being $\theta(0)=1$, we have 
\begin{align*}
||v||_{\mathbb{X}^{s}_{m,T}}^{2}&= C^{2} T^{N}\int_{0}^{+\infty} y^{1-2s} (|\theta'_{m}(y)|^{2}+m^{2} |\theta_{m}(y)|^{2})\, dy\\
&= C^{2} T^{N}m^{2s} \int_{0}^{+\infty} y^{1-2s} (|\theta'(y)|^{2}+m^{2} |\theta(y)|^{2})\, dy\\
&= C^{2} T^{N}m^{2s} \kappa_{s}\\
&= m^{2s} \kappa_{s} |\textup{Tr}(v)|_{L^{2}(0,T)^{N}}^{2}.
\end{align*} 
Now, we assume that 
$$
||v||_{\mathbb{X}^{s}_{m,T}}^{2} - \kappa_{s} m^{2s} |\textup{Tr}(v)|_{L^{2}(0,T)^{N}}^{2} =0.
$$
By using $(ii)$ and $(iii)$ we deduce
\begin{align}
||\textup{Ext}\textup{Tr}(v)||_{\mathbb{X}^{s}_{m,T}}^{2}&\leq 
||v||_{\mathbb{X}^{s}_{m,T}}^{2}= \kappa_{s} m^{2s} |\textup{Tr}(v)|_{L^{2}(0,T)^{N}}^{2} \nonumber \\
&\leq \kappa_{s} |\textup{Tr}(v)|_{\mathbb{H}^{s}_{m,T}}^{2}=||\textup{Ext}\textup{Tr}(v)||_{\mathbb{X}^{s}_{m,T}}^{2}
\end{align}
that is 
$$
||v||_{\mathbb{X}^{s}_{m,T}}^{2}=||\textup{Ext}\textup{Tr}(v)||_{\mathbb{X}^{s}_{m,T}}^{2}=\kappa_{s} m^{2s} |\textup{Tr}(v)|_{L^{2}(0,T)^{N}}^{2}. 
$$
Let us note that $||v||_{\mathbb{X}^{s}_{m,T}}^{2}=||\textup{Ext}\textup{Tr}(v)||_{\mathbb{X}^{s}_{m,T}}^{2}$ implies $v=\textup{Ext}(\textup{Tr}(v))$.

In particular, from
$$
\kappa_{s}m^{2s} |\textup{Tr}(v)|_{L^{2}(0,T)^{N}}^{2}=||\textup{Ext}\textup{Tr}(v)||_{\mathbb{X}^{s}_{m,T}}^{2}=\kappa_{s} |\textup{Tr}(v)|_{\mathbb{H}^{s}_{m,T}}^{2}
$$
we obtain that $c_{k}=0$  for any $k\neq 0$, so we get 
$$
v=\textup{Ext}(\textup{Tr}(v))=\sum_{k\in \Z^{N}} c_{k}\theta(\sqrt{\omega^{2} |k|^{2}+ m^{2}} y) e^{\imath k \cdot x}=c_{0} \theta(my).
$$

\end{proof}

\section{Periodic solutions in the cylinder $\mathcal{S}_{T}$}

In this section we prove the existence of a solution to $(\ref{P})$. As shown in previous section, we know that the study of $(\ref{P})$ is equivalent to investigate the existence of weak solutions to
\begin{equation}
\left\{
\begin{array}{ll}
-\dive(y^{1-2s} \nabla v)+m^{2}y^{1-2s}v =0 &\mbox{ in }\mathcal{S}_{T}:=(0,T)^{N} \times (0,\infty)  \\
v_{| {\{x_{i}=0\}}}= v_{| {\{x_{i}=T\}}} & \mbox{ on } \partial_{L}\mathcal{S}_{T}:=\partial (0,T)^{N} \times [0,\infty) \\
\frac{\partial v}{\partial \nu^{1-2s}}=\kappa_{s} [m^{2s}v+f(x,v)]   &\mbox{ on }\partial^{0}\mathcal{S}_{T}:=(0,T)^{N} \times \{0\}
\end{array}.
\right.
\end{equation}
For simplicity, we will assume that $\kappa_{s}=1$.

Then, we will look for the critical points of
$$
\mathcal{J}_{m}(v)=\frac{1}{2}||v||^{2}_{\mathbb{X}^{s}_{m,T}} -\frac{ m^{2s}}{2}|v(\cdot,0)|_{L^{2}(0,T)^{N}}^{2}- \int_{\partial^{0}\mathcal{S}_{T}} F(x,v) dx
$$
defined for $v\in\mathbb{X}^{s}_{m,T}$.

More precisely, we will prove that $\mathcal{J}_{m}$ satisfies the assumptions of the Linking Theorem \cite{Rab}:
\begin{thm}\label{Linking Thm}
Let $(X, ||\cdot||)$ be a real Banach space with $X=Y\bigoplus Z$, where $Y$ is finite dimensional. Let $J\in \mathcal{C}^{1}(X,\R)$ be a functional satisfying the following conditions:
\begin{itemize}
\item $J$ satisfies the Palais-Smale condition,
\item There exist $\eta, \rho>0$ such that 
$$
\inf \{J(v): v\in Z \mbox{ and } ||v||=\eta\} \geq \rho,
$$
\item There exist $z\in \partial B_{1} \cap Z$, $R>\eta$ and $R'>0$ such that
$$
 J\leq 0 \mbox{ on } \partial \mathcal{A}
$$
where 
$$
\mathcal{A}=\{v=y+r z: y\in Y, ||y||\leq R' \mbox{ and } 0\leq r \leq R\}
$$
and
$$
\partial \mathcal{A}=\{v=y+r z: y\in Y, ||y||= R' \mbox{ or }  r\in \{0, R\} \}.
$$
\end{itemize}
Then $J$ possesses a critical value $c\geq \rho$ which can be characterized as
$$
c:=\inf_{\gamma \in \Gamma} \max_{v\in \mathcal{A}} J(\gamma(v))
$$
where 
$$
\Gamma:=\{\gamma \in \mathcal{C}(\mathcal{A}, X): \gamma=Id \mbox{ on } \partial \mathcal{A}\}.
$$
\end{thm}

By using the assumptions on $f$, it is easy to prove that  $\mathcal{J}_{m}$ is well defined on $\mathbb{X}^{s}_{m,T}$ and $\mathcal{J}_{m}\in \mathcal{C}^{1}(\mathbb{X}^{s}_{m,T},\R)$. Moreover, by (\ref{eqYY}) we notice that the quadratic part of $\mathcal{J}_{m}$ is nonnegative, that is
\begin{equation}\label{partequad}
||v||^{2}_{\mathbb{X}^{s}_{m,T}} -m^{2s}|v(\cdot,0)|_{L^{2}(0,T)^{N}}^{2}\geq 0. 
\end{equation}

Let us note that
$$
\mathbb{X}^{s}_{m,T}=<\theta(my)> \oplus \Bigl\{ v\in \mathbb{X}^{s}_{m,T}: \int_{(0,T)^{N}} v(x,0) dx=0\Bigr \}=:\mathbb{Y}^{s}_{m,T} \oplus \mathbb{Z}^{s}_{m,T}
$$
where $\dim \mathbb{Y}^{s}_{m,T}<\infty$ and $\mathbb{Z}^{s}_{m,T}$ is the orthogonal complement of $\mathbb{Y}^{s}_{m,T}$ with respect to the inner product in $\mathbb{X}^{s}_{m,T}$.
In order to prove that $\mathcal{J}_{m}$ verifies the Linking hypotheses we prove the following Lemmas
\begin{lem}\label{Linking1}
$\mathcal{J}_{m}\leq 0 \mbox{ on } \mathbb{Y}^{s}_{m,T}$.
\end{lem}
\begin{proof}
It follows directly by (\ref{eqY}) and by assumption $(f6)$.

\end{proof}

\begin{lem}
There exist $\rho>0$ and $\eta>0$ such that
$\mathcal{J}_{m}(v)\geq \rho \mbox{ for } v\in \mathbb{Z}^{s}_{m,T}: ||v||_{\mathbb{X}^{s}_{m,T}}=\eta$.
\end{lem}
\begin{proof}
Firstly we show that there exists a constant $C>0$ such that
\begin{align}\label{eqnorm}
||v||_{\mathbb{X}^{s}_{m,T}}^{2}-m^{2s}|v(\cdot,0)|_{L^{2}(0,T)^{N}}^{2}\geq C||v||_{\mathbb{X}^{s}_{m,T}}^{2}
\end{align}
for any $v\in \mathbb{Z}^{s}_{m,T}$.
Assume by contradiction that there exists a sequence $(v_{j})\subset \mathbb{Z}^{s}_{m,T}$ such that
$$
||v_{j}||_{\mathbb{X}^{s}_{m,T}}^{2}-m^{2s}|v_{j}(\cdot,0)|_{L^{2}(0,T)^{N}}^{2}< \frac{1}{j} ||v_{j}||_{\mathbb{X}^{s}_{m,T}}^{2}
$$
Let $z_{j}=v_{j}/||v_{j}||_{\mathbb{X}^{s}_{m,T}}$. Then $||z_{j}||_{\mathbb{X}^{s}_{m,T}}=1$ so we can assume that $z_{j}\rightharpoonup z$ in $\mathbb{X}^{s}_{m,T}$ and $z_{j}(\cdot,0) \rightarrow z(\cdot,0)$ in $L^{2}(0,T)^{N}$ for some $z\in \mathbb{Z}^{s}_{m,T}$ ($\mathbb{Z}^{s}_{m,T}$ is weakly closed).

Hence, for any $j\in \N$
$$
1-m^{2s}|z_{j}(\cdot,0)|_{L^{2}(0,T)^{N}}^{2}<\frac{1}{j}
$$
so  we get $|z_{j}(\cdot,0)|^{2}_{L^{2}(0,T)^{N}}\rightarrow \frac{1}{m^{2s}}$ that is $|z(\cdot,0)|_{L^{2}(0,T)^{N}}=\frac{1}{m^{s}}$.

On the other hand
\begin{align*}
0&\leq ||z||_{\mathbb{X}^{s}_{m,T}}^{2}-m^{2s}|z(\cdot,0)|_{L^{2}(0,T)^{N}}^{2}\\
&\leq \liminf_{j\rightarrow \infty}||z_{j}||_{\mathbb{X}^{s}_{m,T}}^{2}-m^{2s}|z_{j}(\cdot,0)|_{L^{2}(0,T)^{N}}^{2}=0
\end{align*}
implies that $z=c\, \theta(my)$ by (\ref{eqY}). But $z\in \mathbb{Z}^{s}_{m,T}$ so $c=0$ and this is a contradiction because of $|z(\cdot,0)|_{L^{2}(0,T)^{N}}=\frac{1}{m^{s}}>0$.

Taking into account $(\ref{eqnorm})$, $(\ref{F})$ and Theorem $4$ we have
\begin{align*}
\mathcal{J}_{m}(v)&\geq C||v||_{\mathbb{H}^{s}_{m,T}}^{2}-\varepsilon|v(\cdot,0)|_{L^{2}(0,T)^{N}}^{2}-C_{\varepsilon} |v(\cdot,0)|_{L^{p+1}(0,T)^{N}}^{p+1} \\
&\geq \Bigl(C-\frac{\varepsilon}{m}\Bigr)||v||_{\mathbb{X}^{s}_{m,T}}^{2}-C||v||_{\mathbb{X}^{s}_{m,T}}^{p+1}
\end{align*}
for any $v\in \mathbb{Z}^{s}_{m,T}$.
Choosing $\varepsilon \in (0,mC)$, we can find $\rho>0$ and $\eta>0$ such that
$$
\inf \{\mathcal{J}_{m}(v):  v\in \mathbb{Z}^{s}_{m,T} \mbox{ and } ||v||_{\mathbb{X}^{s}_{m,T}}=\eta  \}\geq \rho.
$$

\end{proof}

\begin{lem}\label{lemma5}
There exist $R>\rho$, $R'>0$ and $z\in \mathbb{Z}^{s}_{m,T}$ such that
$$
\max_{\partial \mathbb{A}^{s}_{m,T}} \mathcal{J}_{m}(v) \leq 0 \quad \mbox{ and } \quad \max_{ \mathbb{A}^{s}_{m,T}} \mathcal{J}_{m}(v) <\infty
$$
where
$$
\mathbb{A}^{s}_{m,T}=\{ v=y+r z: ||y||_{\mathbb{X}^{s}_{m,T}} \leq R' \mbox{ and }  r\in [0, R] \}. 
$$
\end{lem}
\begin{proof}
By using Lemma \ref{Linking1} we know that $\mathcal{J}_{m}\leq 0$ on $\mathbb{Y}^{s}_{m,T}$. 

Let us consider
$$
w=\prod_{i=1}^{N} \sin(\omega x_{i}) \frac{1}{y+1}.
$$
We note that $w\in \mathbb{Z}^{s}_{m,T}$ (since $\int_{0}^{T} \sin(\omega x) dx=0$) and 
\begin{align*}
||w||^{2}_{\mathbb{X}^{s}_{m,T}}&=N\Bigl(\prod_{i=1}^{N-1} \int_{0}^{T} \sin^{2}(\omega x) dx\Bigr)
\omega^{2} \Bigl( \int_{0}^{T} \cos^{2}(\omega x) dx\Bigr)\Bigl(\int_{0}^{\infty} y^{1-2s} \frac{dy}{(y+1)^{2}} \Bigr) \\
&+\Bigl(\prod_{i=1}^{N} \int_{0}^{T} \sin^{2}(\omega x) dx \Bigr) \Bigl(\int_{0}^{\infty} y^{1-2s} \frac{dy}{(y+1)^{4}} \Bigr)\\
&+m^{2}\Bigl(\prod_{i=1}^{N} \int_{0}^{T} \sin^{2}(\omega x) dx \Bigr) \Bigl(\int_{0}^{\infty} y^{1-2s} \frac{dy}{(y+1)^{2}} \Bigr).
\end{align*}
So there exist $C_{1}, C_{2}, C_{3}>0$ (independent of $m$) such that 
\begin{equation}\label{einstein}
C_{1}\leq ||w||^{2}_{\mathbb{X}^{s}_{m,T}}\leq C_{2}+m^{2}C_{3}.
\end{equation}
We set  $z=\frac{w}{||w||_{\mathbb{X}^{s}_{m,T}}}$.
It is clear that $z\in \mathbb{Z}^{s}_{m,T}$ and $||z||_{\mathbb{X}^{s}_{m,T}}=1$.

By H\"older inequality, we can observe that if $v=y+rz\in \mathbb{Y}^{s}_{m,T}\oplus \R_{+}z$
\begin{align*}
|v(\cdot,0)|_{L^{\mu}(0,T)^{N}}^{\mu} &\geq C|v(\cdot,0)|_{L^{2}(0,T)^{N}}^{\mu} \\
&=C\Bigl(\int_{(0,T)^{N}} (c+rz)^{2} dx \Bigr)^{\frac{\mu}{2}} \\
& \geq  C'(m^{2s}c^{2}T^{N}+r^{2})^{\frac{\mu}{2}}. 
\end{align*}

Then, for any $v=y+r z\in \mathbb{Y}^{s}_{m,T}\oplus \R_{+} z$
\begin{align}
\mathcal{J}_{m}(v)&=\frac{1}{2}||z||_{\mathbb{X}^{s}_{m,T}}^{2}-\frac{m^{2s}}{2}|z(\cdot,0)|_{L^{2}(0,T)^{N}}^{2}-\int_{\partial^{0}\mathcal{S}_{T}} F(x,v) dx \nonumber \\
&\leq \frac{r^{2}}{2}-A|v(\cdot,0)|_{L^{\mu}(0,T)^{N}}^{\mu}+BT^{N} \nonumber \\
&\leq \frac{r^{2}}{2}-C''(m^{2s}c^{2}T^{N}+r^{2})^{\frac{\mu}{2}} +BT^{N} \label{4.12} \\
&\leq (m^{2s}c^{2}T^{N}+r^{2})-C''(m^{2s}c^{2}T^{N}+r^{2})^{\frac{\mu}{2}} +BT^{N} \label{4.13}\\
&=||v||_{\mathbb{X}^{s}_{m,T}}^{2}-E||v||_{\mathbb{X}^{s}_{m,T}}^{\mu}+F. \label{4.14}
\end{align}

We recall that $\mu>2$.
By using (\ref{4.12}), there exists $R>0$ such that 
$$
\mathcal{J}_{m}(y+rz) \leq 0 \mbox{ for any } r \geq  R \mbox{ and } y \in \mathbb{Y}^{s}_{m,T}.
$$
Let $r\in [0,R]$. By (\ref{4.13}), we can find $R'>0$ such that $\mathcal{J}_{m}(y+rz) \leq 0$ for $||y||_{\mathbb{X}^{s}_{m,T}} \geq R'$. By (\ref{4.14}), we deduce that there exists a constant $\delta>0$ such that $\mathcal{J}_{m}(v)\leq \delta$ for any $v\in \mathbb{A}^{s}_{m,T}$.

\end{proof}

Finally we show that $\mathcal{J}_{m}$ satisfies the Palais-Smale condition:

\begin{lem}\label{Linking4}
Let $c \in \R$. Let $(v_{j})\subset \mathbb{X}^{s}_{m,T}$ be a sequence such that 
\begin{equation}\label{PS}
\mathcal{J}_{m}(v_{j}) \rightarrow c \mbox{ and }  \mathcal{J}_{m}'(v_{j}) \rightarrow 0.
\end{equation}
Then there exist a subsequence $(v_{j_{h}})\subset (v_{j})$ and $v \in \mathbb{X}^{s}_{m,T}$ such that $v_{j_{h}} \rightarrow v$ in $\mathbb{X}^{s}_{m,T}$.
\end{lem}

\begin{proof}
We start proving that $(v_{j})$ is bounded in $\mathbb{X}^{s}_{m,T}$.

Fix $\beta \in (\frac{1}{\mu}, \frac{1}{2})$. 
By using Lemma \ref{lfF} with $\varepsilon=1$, we get
\begin{align}\label{teresa1}
\Bigl|\int_{\partial^{0}\mathcal{S}_{T}\cap \{|v_{j}|\leq r_{0}\}} \beta f(x,v_{j})v_{j}-F(x,v_{j}) dx \Bigr| &\leq ((2\beta +1) r_{0}^{2} + C_{1} (p+2) r_{0}^{p+1}) T^{N}= \iota_{1}
\end{align}
and
\begin{align}\label{teresa2}
\Bigl|\int_{\partial^{0}\mathcal{S}_{T}\cap \{|v_{j}|\leq r_{0}\}} F(x,v_{j}) dx \Bigr| &\leq (r_{0}^{2}+ C_{1} r_{0}^{p+1}) T^{N}= \iota_{2}. 
\end{align}
Taking into account Lemma \ref{F**}, $(f5)$, (\ref{partequad}), (\ref{PS}), (\ref{teresa1}) and (\ref{teresa2}) we have for $j$ large 
\begin{align}
&c+1+||v_{j}||_{\mathbb{X}^{s}_{m,T}} \geq  \mathcal{J}_{m}(v_{j})-\beta \langle  \mathcal{J}'(v_{j}), v_{j} \rangle \nonumber \\
&=\Bigl(\frac{1}{2}-\beta \Bigr) \Bigl[ ||v_{j}||_{\mathbb{X}^{s}_{m,T}}^{2}-m^{2s}|v_{j}(\cdot,0)|_{L^{2}(0,T)^{N}}^{2} \Bigr]+\int_{\partial^{0}\mathcal{S}_{T}} \beta f(x,v_{j})v_{j}-F(x,v_{j}) dx \nonumber \\
&\geq \int_{\partial^{0} \mathcal{S}_{T}} [\beta f(x,v_{j})v_{j}-F(x,v_{j})]dx \nonumber \\
&= \int_{\partial^{0} \mathcal{S}_{T}\cap \{|v_{n}|\geq r_{0}\}} [\beta f(x,v_{j})v_{j}-F(x,v_{j})]dx  +  \int_{\partial^{0} \mathcal{S}_{T}\cap \{|v_{n}|\leq r_{0}\}} [\beta f(x,v_{j})v_{j}-F(x,v_{j})]dx  \nonumber \\
&\geq (\mu \beta-1) \int_{\partial^{0}\mathcal{S}_{T}\cap \{|v_{j}|\geq r_{0}\}} F(x,v_{j}) dx-\iota_{1} \nonumber\\
&\geq (\mu \beta -1)\int_{\partial^{0} \mathcal{S}_{T}} F(x,v_{j})dx -(\mu \beta -1)\iota_{2}  -\iota_{1} \nonumber \\
&= (\mu \beta-1) \int_{\partial^{0}\mathcal{S}_{T}} F(x,v_{j}) dx-\iota \label{4.2}\\
& \geq (\mu \beta-1)[a_{3}|v_{j}(\cdot,0)|_{L^{\mu}(0,T)^{N}}^{\mu}-a_{4}T^{N}]-\iota \nonumber \\
& \geq (\mu \beta-1)[a_{3}|v_{j}(\cdot,0)|^{\mu}_{L^{2}(0,T)^{N}}T^{-N\frac{\mu-2}{2}}-a_{4}T^{N}]-\iota. \label{4.3}
\end{align}

Hence, by (\ref{4.2}) and (\ref{4.3}), we deduce that
\begin{align*}
||v_{j}||_{\mathbb{X}^{s}_{m,T}}^{2} &=2 \mathcal{J}_{m}(v_{j})+m^{2s}|v_{j}(\cdot,0)|^{2}_{L^{2}(0,T)^{N}}+2\int_{\partial^{0}\mathcal{S}_{T}} F(x,v_{j}) dx \\
& \leq C_{1}+C_{2} (C_{3}+1+||v_{j}||_{\mathbb{X}^{s}_{m,T}})^{\frac{2}{\mu}}+ C_{4}(C_{5} +1+||v_{j}||_{\mathbb{X}^{s}_{m,T}}) \\
& \leq C_{6}+ C_{7} ||v_{j}||_{\mathbb{X}^{s}_{m,T}}
\end{align*}
that is $(v_{j})$ is bounded in $\mathbb{X}^{s}_{m,T}$.

By Theorem \ref{tracethm} we can assume, up to a subsequence,  that
\begin{align}\label{3.7}
v_{j} &\rightharpoonup v \mbox{ in } \mathbb{X}^{s}_{m,T}  \nonumber \\
v_{j}(\cdot,0) &\rightarrow v(\cdot,0) \mbox{ in } L^{p+1}(0,T)^{N} \\
v_{j}(\cdot,0) &\rightarrow  v(\cdot,0) \mbox{ a.e in } (0,T)^{N} \nonumber
\end{align}
as $j\rightarrow \infty$ and there exists $h\in L^{p+1}(0,T)^{N}$ such that
\begin{equation}\label{3.10}
|v_{j}(x,0)|\leq h(x) \quad \mbox{ a.e. in } x\in (0,T)^{N}, \quad \mbox{ for all } j\in \N.
\end{equation}
Taking into account $(f2)$, $(f4)$, (\ref{3.7}), (\ref{3.10}) and the Dominated Convergence Theorem  we get
\begin{equation}\label{1}
\int_{\partial^{0}\mathcal{S}_{T}} f(x,v_{j}) v_{j} dx \rightarrow \int_{\partial^{0}\mathcal{S}_{T}} f(x,v) v dx
\end{equation}
and
\begin{equation}\label{2}
\int_{\partial^{0}\mathcal{S}_{T}} f(x,v_{j}) v dx \rightarrow \int_{\partial^{0}\mathcal{S}_{T}} f(x,v) v dx
\end{equation}
as $j \rightarrow \infty$.

By using (\ref{PS}) and the boundedness of $(v_{j})_{j\in \N}$ in $\mathbb{X}^{s}_{m,T}$, we deduce that 
$\langle  \mathcal{J}_{m}'(v_{j}),v_{j} \rangle \rightarrow 0$, that is
\begin{equation}\label{wc1}
||v_{j}||_{\mathbb{X}^{s}_{m,T}}^{2}- m^{2s}|v_{j}(\cdot,0)|_{L^{2}(0,T)^{N}}^{2}-\int_{\partial^{0}\mathcal{S}_{T}} f(x,v_{j}) v_{j} dx \rightarrow 0
\end{equation}
as $j\rightarrow \infty$.
By $(\ref{3.7})$, (\ref{1}) and (\ref{wc1}) we have
\begin{equation}\label{wc2}
||v_{j}||_{\mathbb{X}^{s}_{m,T}}^{2} \rightarrow m^{2s}|v(\cdot,0)|_{L^{2}(0,T)^{N}}^{2}-\int_{\partial^{0}\mathcal{S}_{T}} f(x,v) v dx.
\end{equation}
Moreover, by (\ref{PS}) and $v \in \mathbb{X}^{s}_{m,T}$, we have $\langle  \mathcal{J}_{m}'(v_{j}),v \rangle \rightarrow 0$ as $j\rightarrow \infty$, that is
\begin{equation}\label{wc3}
\langle v_{j}, v \rangle_{\mathbb{X}^{s}_{m,T}}- m^{2s} \langle v_{j},v \rangle_{L^{2}(0,T)^{N}}-\int_{\partial^{0}\mathcal{S}_{T}} f(x,v_{j}) v dx \rightarrow 0
\end{equation}
Taking into account (\ref{3.7}), (\ref{3.10}), (\ref{2}) and (\ref{wc3}) we obtain
\begin{equation}\label{wc4}
||v||_{\mathbb{X}^{s}_{m,T}}^{2}= m^{2s}|v(\cdot,0)|_{L^{2}(0,T)^{N}}^{2}-\int_{\partial^{0}\mathcal{S}_{T}} f(x,v) v dx.
\end{equation}
Thus, (\ref{wc2}) and (\ref{wc4}) imply that 
\begin{equation}\label{wc5}
||v_{j}||_{\mathbb{X}^{s}_{m,T}}^{2}\rightarrow ||v||_{\mathbb{X}^{s}_{m,T}}^{2} \mbox{ as } j\rightarrow \infty.
\end{equation}
Since $\mathbb{X}^{s}_{m,T}$ is a Hilbert space, we have
$$
||v_{j}-v||_{\mathbb{X}^{s}_{m,T}}^{2}=||v_{j}||_{\mathbb{X}^{s}_{m,T}}^{2}+||v||_{\mathbb{X}^{s}_{m,T}}^{2}-2 \langle v_{j},v \rangle_{\mathbb{X}^{s}_{m,T}}
$$
and using $v_{j} \rightharpoonup v$ in $\mathbb{X}^{s}_{m,T}$ and (\ref{wc5})
we can conclude that $v_{j} \rightarrow v$ in $\mathbb{X}^{s}_{m,T}$, as $j \rightarrow \infty$.

\end{proof}

\begin{proof} [of Theorem \ref{thm1}]
Taking into account Lemma \ref{Linking1} - Lemma \ref{Linking4}, by the Theorem \ref{Linking Thm} we deduce that for any fixed $m>0$, there exists of a function $v\in \mathbb{X}^{s}_{m,T}$ such that
$$
\mathcal{J}_{m}(v_{m})=\alpha_{m}, \quad \mathcal{J}'_{m}(v_{m})=0
$$
where
\begin{equation}\label{criticalvalue}
\alpha_{m}=\inf_{\gamma \in \Gamma_{m}} \max_{v\in \mathbb{A}^{s}_{m,T}} \mathcal{J}_{m}(\gamma(v))
\end{equation}
and 
$$
\Gamma_{m}=\{\gamma \in \mathcal{C}(\mathbb{A}^{s}_{m,T}, \mathbb{X}^{s}_{m,T}): \gamma=Id \mbox{ on } \partial \mathbb{A}^{s}_{m,T}\}.
$$

\end{proof}

\begin{remark}
Let us observe that an easy consequence of Theorem \ref{thm1} is the existence of infinitely many distinct $T$-periodic solutions to (\ref{P}). To prove it, one can proceed as in the proof of Corollary $6.44$ in \cite{Rab}.
\end{remark}

\section{Regularity of solutions to $(\ref{P})$}

In this section we study the regularity of weak solutions to the problem $(\ref{P})$.

Firstly we prove the following 

\begin{lem}\label{lemmino}

Let $v\in \mathbb{X}^{s}_{m,T}$ be a weak solution to 
\begin{equation}
\left\{
\begin{array}{ll}
-\dive(y^{1-2s} \nabla v)+m^{2}y^{1-2s}v =0 &\mbox{ in }\mathcal{S}_{T} \\
v_{| {\{x_{i}=0\}}}= v_{| {\{x_{i}=T\}}} &\mbox{ on } \partial_{L}\mathcal{S}_{T} \\
\frac{\partial v}{\partial \nu^{1-2s}}=m^{2s}v+f(x,v)  &\mbox{ on } \partial^{0} \mathcal{S}_{T} 
\end{array}.
\right.
\end{equation}
Then $v(\cdot,0)\in L^{q}(0,T)^{N}$ for all $q<\infty$.
\end{lem}
\begin{proof}
We proceed as in the proof of Lemma $7$ in \cite{A2}. 
Since $v$ is a critical point for $\mathcal{J}_{m}$, we know that
\begin{equation}\label{criticpoint}
\iint_{\mathcal{S}_{T}} y^{1-2s}(\nabla v \nabla \eta+m^{2}v \eta) \, dxdy=\int_{\partial^{0} \mathcal{S}_{T}} m^{2s}v\eta+f(x,v)\eta \,dx
\end{equation}
for all $\eta\in \mathbb{X}^{m}_{T}$.

Let $w=vv^{2\beta}_{K}\in \mathbb{X}^{s}_{m,T}$ where $v_{K}=\min\{|v|,K\}$, $K>1$ and $\beta\geq 0$.
Taking $\eta=w$ in (\ref{criticpoint}) we deduce that 
\begin{align}\label{conto1}
\iint_{\mathcal{S}_{T}}  &y^{1-2s}v^{2\beta}_{K}(|\nabla v|^{2}+m^{2}v^{2}) \, dxdy+\iint_{D_{K,T}} 2\beta y^{1-2s}v^{2\beta}_{K} |\nabla v|^{2} \, dx dy  \nonumber \\
&=m^{2s}\int_{\partial^{0}\mathcal{S}_{T}} v^{2} v^{2\beta}_{K} \,dx+ \int_{\partial^{0}\mathcal{S}_{T}} f(x,v)vv^{2\beta}_{K} \,dx 
\end{align}
where $D_{K,T}=\{(x,y)\in \mathcal{S}_{T}: |v(x,y)|\leq K\}$. \\
It is easy to see that
\begin{align}\label{conto2}
\iint_{\mathcal{S}_{T}} &y^{1-2s}|\nabla (vv_{K}^{\beta})|^{2} dxdy \nonumber \\
&=\iint_{\mathcal{S}_{T}} y^{1-2s}v_{K}^{2\beta} |\nabla v|^{2} dxdy+\iint_{D_{K,T}} (2\beta+\beta^{2}) y^{1-2s}v_{K}^{2\beta} |\nabla v|^{2} dxdy.
\end{align}
Then, putting together (\ref{conto1}) and (\ref{conto2}) we get 
\begin{align}\label{S1}
&||vv_{K}^{\beta}||_{\mathbb{X}^{s}_{m,T}}^{2} \nonumber \\
&=\iint_{\mathcal{S}_{T}} y^{1-2s}[|\nabla (vv_{K}^{\beta})|^{2}+m^{2}v^{2}v_{K}^{2\beta}] dxdy \nonumber \\
&=\iint_{\mathcal{S}_{T}} y^{1-2s}v_{K}^{2\beta}[ |\nabla v|^{2}+m^{2}v^{2}] dxdy+\iint_{D_{K,T}} 2\beta \Bigl(1+\frac{\beta}{2}\Bigr) y^{1-2s}v_{K}^{2\beta} |\nabla v|^{2} dxdy \nonumber \\
&\leq c_{\beta} \Bigl[\iint_{\mathcal{S}_{T}} y^{1-2s}v_{K}^{2\beta}[ |\nabla v|^{2}+m^{2}v^{2}] dxdy+\iint_{D_{K,T}} 2\beta y^{1-2s}v_{K}^{2\beta} |\nabla v|^{2} dxdy\Bigr] \nonumber \\
&=c_{\beta} \int_{\partial^{0}\mathcal{S}_{T}} m^{2s}v^{2}v_{K}^{2\beta} + f(x,v)v v_{K}^{2\beta} \,dx 
\end{align}
where $c_{\beta}=1+\frac{\beta}{2}$.
By Lemma \ref{lfF} with $\varepsilon=1$ we deduce that
\begin{equation*}
m^{2s}v^{2}v_{K}^{2\beta} + f(x,v)v v_{K}^{2\beta}\leq (m^{2s}+2)v^{2}v_{K}^{2\beta}+(p+1)C_{1}|v|^{p-1}v^{2}v_{K}^{2\beta} \mbox{ on } \partial^{0}\mathcal{S}_{T}.
\end{equation*}
Now, we prove that
\begin{equation*}
|v|^{p-1}\leq 1+h \mbox{ on } \partial^{0}\mathcal{S}_{T}
\end{equation*}
for some $h\in L^{N/2s}(0,T)^{N}$.
Firstly, we observe that
$$
|v|^{p-1}=\chi_{\{|v|\leq 1\}}|v|^{p-1}+\chi_{\{|v|>1\}}|v|^{p-1}\leq 1+\chi_{\{|v|>1\}}|v|^{p-1}\ \mbox{ on } \partial^{0}\mathcal{S}_{T}.
$$
If $(p-1)N<4s$ then 
$$
\int_{\partial^{0}\mathcal{S}_{T}} \chi_{\{|v|>1\}}|v|^{\frac{N}{2s}(p-1)} dx \leq \int_{\partial^{0}\mathcal{S}_{T}} \chi_{\{|v|>1\}}|v|^{2} dx<\infty
$$
while if $4s\leq (p-1)N$ we have that $(p-1)\frac{N}{2s}\in [2,\frac{2N}{N-2s}]$.

Therefore, there exist a constant $c=m^{2s}+2+(p+1)C_{1}$ and a function $h\in L^{N/2s}(0,T)^{N}$, $h\geq 0$ and independent of $K$ and $\beta$, such that
\begin{equation}\label{S2}
m^{2s}v^{2}v_{K}^{2\beta} + f(x,v)vv_{K}^{2\beta}\leq (c+h)v^{2}v_{K}^{2\beta} \mbox{ on } \partial^{0}\mathcal{S}_{T}.
\end{equation}
Taking into account (\ref{S1}) and (\ref{S2}) we have
\begin{equation*}
||vv_{K}^{\beta}||_{\mathbb{X}^{s}_{m,T}}^{2}\leq c_{\beta} \int_{\partial^{0}\mathcal{S}_{T}} (c+h)v^{2}v_{K}^{2\beta} dx,
\end{equation*}
and by Monotone Convergence Theorem ($v_{K}$ is increasing with respect to $K$) we have as $K\rightarrow \infty$
\begin{equation}\label{i1}
|||v|^{\beta+1}||_{\mathbb{X}^{s}_{m,T}}^{2}\leq cc_{\beta} \int_{\partial^{0}\mathcal{S}_{T}} |v|^{2(\beta +1)} dx + c_{\beta}\int_{\partial^{0}\mathcal{S}_{T}}  h|v|^{2(\beta +1)}dx.
\end{equation}
Fix $M>0$ and let $A_{1}=\{h\leq M\}$ and $A_{2}=\{h>M\}$.

Then
\begin{equation}\label{i2}
\int_{\partial^{0}\mathcal{S}_{T}}  h|v(\cdot,0)|^{2(\beta +1)} dx\leq M||v(\cdot,0)|^{\beta+1}|_{L^{2}(0,T)^{N}}^{2}+\varepsilon(M)||v(\cdot,0)|^{\beta+1}|_{L^{2^{\s}}(0,T)^{N}}^{2}
\end{equation}
where $\varepsilon(M)=\Bigl(\int_{A_{2}} h^{N/2s} dx \Bigr)^{\frac{2s}{N}}\rightarrow 0$ as $M\rightarrow \infty$.
Taking into account (\ref{i1}), (\ref{i2}), we get
\begin{equation}\label{regv}
|||v|^{\beta+1}||_{\mathbb{X}^{s}_{m,T}}^{2}\leq c_{\beta}(c+M)||v(\cdot,0)|^{\beta+1}|_{L^{2}(0,T)^{N}}^{2}+c_{\beta}\varepsilon(M)||v(\cdot,0)|^{\beta+1}|_{L^{2^{\s}_{s}}(0,T)^{N}}^{2}.
\end{equation}
By using Theorem $\ref{thm2}$ we know that there exists a constant $C^{2}_{2^{\s}_{s},m}>0$ such that
\begin{equation}\label{S3}
||v(\cdot,0)|^{\beta+1}|_{L^{2^{\s}_{s}}(0,T)^{N}}^{2}\leq C^{2}_{2^{\s}_{s},m} |||v|^{\beta+1}||_{\mathbb{X}^{s}_{m,T}}^{2}.
\end{equation}
Then, choosing $M$ large so that $\varepsilon(M) c_{\beta} C^{2}_{2^{\s},m}<\frac{1}{2}$,
and by using $(\ref{regv})$ and $(\ref{S3})$ we obtain
\begin{equation}\label{iter}
||v(\cdot,0)|^{\beta+1}|_{L^{2^{\s}_{s}}(0,T)^{N}}^{2}\leq 2 C^{2}_{2^{\s}_{s},m} c_{\beta}(c+M)||v(\cdot,0)|^{\beta+1}|^{2}_{L^{2}(0,T)^{N}}.
\end{equation}
Then we can start a bootstrap argument: since $v(\cdot,0)\in L^{\frac{2N}{N-2s}}$ we can apply (\ref{iter}) with $\beta_{1}+1=\frac{N}{N-2s}$ to deduce that $v(\cdot,0)\in L^{\frac{(\beta_{1}+1)2N}{N-2s}}(0,T)^{N}=L^{\frac{2N^{2}}{(N-2s)^{2}}}(0,T)^{N}$. Applying (\ref{iter}) again, after $k$ iterations, we find $v(\cdot,0)\in L^{\frac{2N^{k}}{(N-2s)^{k}}}(0,T)^{N}$, and so $v(\cdot,0)\in L^{q}(0,T)^{N}$ for all $q\in[2,\infty)$.

\end{proof}

Then, we can deduce the following result:

\begin{thm}\label{regularity}
Let $v\in \mathbb{X}^{s}_{m,T}$ be a weak solution to 
\begin{equation}
\left\{
\begin{array}{ll}
-\dive(y^{1-2s} \nabla v)+m^{2}y^{1-2s}v =0 &\mbox{ in }\mathcal{S}_{T} \\
v_{| {\{x_{i}=0\}}}= v_{| {\{x_{i}=T\}}} &\mbox{ on } \partial_{L}\mathcal{S}_{T} \\
\frac{\partial v}{\partial \nu^{1-2s}}=\kappa_{s}[m^{2s}v+f(x,v)]  &\mbox{ on } \partial^{0} \mathcal{S}_{T} 
\end{array}.
\right.
\end{equation}
Let us assume that $v$ is extended by periodicity to the whole $\R^{N+1}_{+}$. 
Then $v(\cdot,0)\in \mathcal{C}^{0,\alpha}(\R^{N})$ for some $\alpha \in (0,1)$.
\end{thm}
\begin{proof}
It is clear that $v \in H^{1}_{m}(A\times \R_{+},y^{1-2s})$ for any bounded domain $A\subset \R^{N}$.
By using Lemma \ref{lemmino} here and Proposition $3.5.$  in \cite{FallFelli} we deduce the thesis.

\end{proof}

\section{Passage to the limit as $m \rightarrow 0$}

In this last section, we give the proof of Theorem \ref{thmdue}. We verify that it is possible to take the limit in (\ref{R}) as $m\rightarrow 0$ so that we deduce the existence of a nontrivial weak solution to (\ref{P'}).
In particular, we will prove that such solution is H\"older continuous.
We remark that in Section $4$ we proved that for any $m>0$ there exists $v_{m}\in \mathbb{X}^{s}_{m,T}$ such that 
\begin{align}\label{zero}
\mathcal{J}_{m}(v_{m})= \alpha_{m} \quad \mbox{ and } \quad \mathcal{J}'_{m}(v_{m})=0,
\end{align} 
where $\alpha_{m}$ is defined in (\ref{criticalvalue}). 
In order to attain our aim, we estimate, from above and below, the critical levels of the functional $\mathcal{J}_{m}$ independently of $m$. 

Let us assume that $0<m<m_{0}:=\frac{1}{2C^{2}_{2^{\s}_{s}}}$, where $C_{2^{\s}_{s}}$ is the Sobolev constant which appears in (\ref{benyi}).

We start proving that there exists a positive constant $\delta$ independent of $m$ such that 
\begin{equation}\label{cmdelta}
\alpha_{m}\leq \delta \quad \mbox{ for all }  0<m<m_{0}. 
\end{equation}

By using (\ref{einstein}) and $m<m_{0}$ we know that
$$
C_{1}\leq ||z||_{\mathbb{X}^{s}_{m,T}}^{2} \leq C_{2}+m_{0}^{2}C_{3}.
$$
Moreover (see Lemma \ref{lemma5}) we have for any $v=y+rz \in \mathbb{Y}^{s}_{m,T} \oplus \R_{+}z$
\begin{align}
|v(\cdot,0)|_{L^{\mu}(0,T)^{N}}^{\mu} &\geq T^{-\frac{N(\mu-2)}{2}}|v(\cdot,0)|_{L^{2}(0,T)^{N}}^{\mu} \nonumber \\
&=T^{-\frac{N(\mu-2)}{2}}\Bigl(\int_{(0,T)^{N}} (c+rz)^{2} dx \Bigr)^{\frac{\mu}{2}} \nonumber \\
& \geq  T^{-\frac{N(\mu-2)}{2}} \Bigl(c^{2}T^{N}+(T/2)^{N}\frac{r^{2}}{||z||^{2}_{\mathbb{X}^{s}_{m,T}}} \Bigr)^{\frac{\mu}{2}} \nonumber\\
& \geq  T^{-\frac{N(\mu-2)}{2}} \min\Bigl \{\frac{1}{m^{2s}_{0}}, \frac{(T/2)^{N}}{C_{2}+m_{0}^{2}C_{3}} \Bigr\}(m^{2s}c^{2}T^{N}+r^{2})^{\frac{\mu}{2}} \nonumber \\
&= C||v||_{\mathbb{X}^{s}_{m,T}}^{\mu}
\end{align}
for some $C=C(m_{0},T,N,s)>0$.

Then, for any $v=y+rz\in \mathbb{Y}^{s}_{m,T}\oplus \R_{+} z$ and $0<m<m_{0}$ we get
\begin{align}
\mathcal{J}_{m}(v)&=\frac{1}{2}||v||_{\mathbb{X}^{s}_{m,T}}^{2}-\frac{m^{2s}}{2}|v(\cdot,0)|_{L^{2}(0,T)^{N}}^{2}-\int_{\partial^{0} \mathcal{S}_{T}} F(x,v) dx \nonumber\\
&\leq \frac{1}{2}||v||_{\mathbb{X}^{s}_{m,T}}^{2} -A|v(\cdot,0)|_{L^{\mu}(0,T)^{N}}^{\mu}+B\, T^{N} \nonumber \\
&=||v||_{\mathbb{X}^{s}_{m,T}}^{2}-C||v||_{\mathbb{X}^{s}_{m,T}}^{\mu}+D\leq \delta
\end{align}
where $A, B, C, D, \delta>0$ are independent of $m$.

Now we prove that there exists $\lambda>0$ independent of $m$, such that
\begin{equation}\label{cmlambda}
\alpha_{m}\geq \lambda \quad \mbox{ for all }  0<m<m_{0}. 
\end{equation}

Let $v\in \mathbb{Z}^{s}_{m,T}$ and $\varepsilon>0$. We denote by $c_{k}$ the Fourier coefficients of the trace of $v$.


By using (\ref{benyi}) and (\ref{eqYY}) (with $\kappa_{s}=1$), we know that 
\begin{align}
|v|_{L^{q}(0,T)^{N}} &\leq C_{2^{\s}_{s}} \Bigl( \sum_{|k|\geq 1} \omega ^{2s} |k|^{2s} |c_{k}|^{2} \Bigr)^{\frac{1}{2}} \nonumber \\
&\leq C_{2^{\s}_{s}}  |v|_{\mathbb{H}^{s}_{m,T}} \nonumber \\
&\leq C_{2^{\s}_{s}}  ||v||_{\mathbb{X}^{s}_{m,T}} \label{benyioh}
\end{align}
for any $q\in [2, 2^{\s}_{s}]$. 

By using Lemma \ref{lfF} and (\ref{benyioh}) we can see that for every $0<m<m_{0}$

\begin{align*}
\mathcal{J}_{m}(v)&=\frac{1}{2}\iint_{\mathcal{S}_{T}} y^{1-2s}(|\nabla v|^{2}+m^{2} v^{2}) \ dx dy-\frac{m^{2s}}{2}\int_{\partial^{0} \mathcal{S}_{T}} |v|^{2} dx -\int_{\partial^{0} \mathcal{S}_{T}} F(x,v) dx \\
&\geq \frac{1}{2} ||v||^{2}_{\mathbb{X}^{s}_{m,T}}-\Bigl(\frac{m}{2}+\varepsilon \Bigr) |v(\cdot,0)|_{L^{2}(0,T)^{N}}^{2} -C_{\varepsilon} |v(\cdot,0)|_{L^{p+1}(0,T)^{N}}^{p+1} \\
&\geq  \Bigl[\frac{1}{2}-C^{2}_{2^{\s}_{s}} \Bigl(\frac{m}{2}+\varepsilon \Bigr) \Bigr] ||v||^{2}_{\mathbb{X}^{m}_{T}} -C_{\varepsilon} C^{p+1}_{2^{\s}_{s}} ||v||_{\mathbb{X}^{s}_{m,T}}^{p+1}    \\
& \geq \Bigl(\frac{1}{4}- C^{2}_{2^{\s}_{s}}\varepsilon   \Bigr) ||v||^{2}_{\mathbb{X}^{s}_{m,T}} -C'_{\varepsilon}||v||_{\mathbb{X}^{s}_{m,T}}^{p+1} .
\end{align*}

Choosing $0<\varepsilon<\frac{1}{4 C^{2}_{2^{\s}_{s}}}$, we have that $b:=\frac{1}{4}-C^{2}_{2^{\s}_{s}}\varepsilon>0$.\\
Let $\rho:=\Bigl(\frac{b}{2C'_{b}}\Bigr)^{\frac{1}{p-1}}$.  Then, for every $v\in \mathbb{Z}^{s}_{m,T}$ such that $||v||_{\mathbb{X}_{m,T}^{s}}=\rho$
$$
\mathcal{J}_{m}(v) \geq  b\rho^{2}-C'_{b}\, \rho^{p+1}= \frac{b}{2} \Bigl(\frac{b}{2C'_{b}} \Bigr)^{\frac{2}{p-1}}=:\lambda .
$$
Therefore, taking into account (\ref{cmdelta}) and (\ref{cmlambda}) we deduce that 
\begin{equation}\label{alpham}
\lambda \leq \alpha_{m}\leq \delta \,  \mbox{ for every } \,  0<m<m_{0}.
\end{equation}

Now, we estimate the $H^{1}_{loc}(\mathcal{S}_{T}, y^{1-2s})$- norm of $v_{m}$ in order to pass to the limit in $(\ref{R})$ as $m \rightarrow 0$.

Fix $\beta \in (\frac{1}{\mu},\frac{1}{2})$.
By using  (\ref{zero}) and (\ref{alpham}), we have for any $m\in (0, m_{0})$
\begin{align}
\delta &\geq \mathcal{J}_{m}(v_{m})-\beta \langle \mathcal{J}_{m}'(v_{m}),v_{m} \rangle  \nonumber \\
&=\Bigl(\frac{1}{2}-\beta\Bigr) [ ||v_{m}||_{\mathbb{X}^{s}_{m,T}}^{2}- m^{2s}|v_{m}(\cdot,0)|_{L^{2}(0,T)^{N}}^{2} ] \\
&+\int_{\partial^{0} \mathcal{S}_{T}} [\beta f(x,v_{m})v_{m}-F(x,v_{m})]dx  \nonumber \\
&\geq \int_{\partial^{0} \mathcal{S}_{T}} [\beta f(x,v_{m})v_{m}-F(x,v_{m})]dx  \nonumber \\
&\geq (\mu \beta -1)\int_{\partial^{0} \mathcal{S}_{T}} F(x,v_{m})dx -\tilde{\kappa}   \label{Fineq} \\
&\geq (\mu \beta-1)[a_{3}|v_{m}(\cdot,0)|^{\mu}_{L^{\mu}(0,T)^{N}}-a_{4}T^{N}]-\tilde{\kappa} \nonumber  \\
&\geq (\mu \beta-1)[a_{3}|v_{m}(\cdot,0)|^{\mu}_{L^{2}(0,T)^{N}}T^{-N\frac{\mu-2}{2}}-a_{4}T^{N}]-\tilde{\kappa}. \label{Fi}
\end{align}
By (\ref{Fi}) we deduce that the trace of $v_{m}$ is bounded in $L^{2}(0,T)^{N}$
\begin{equation}\label{vti}
|v_{m}(\cdot,0)|_{L^{2}(0,T)^{N}} \leq K(\delta) \mbox{ for every } m\in (0,m_{0}).
\end{equation}
Taking into account (\ref{zero}), (\ref{alpham}), (\ref{Fineq}) and (\ref{vti}) we deduce
\begin{align}\label{nablavti}
||\nabla v_{m}||^{2}_{L^{2}(\mathcal{S}_{T},y^{1-2s})}&\leq ||v_{m}||^{2}_{\mathbb{X}^{s}_{m,T}} \nonumber \\
&=2J_{m}(v_{m})+m|v_{m}(\cdot,0)|_{L^{2}(0,T)^{N}}^{2}+2\int_{\partial^{0} \mathcal{S}_{T}} F(x,v_{m})dx \nonumber  \\
&\leq 2\delta+\frac{\omega}{2}K(\delta)+C(\delta)=:K'(\delta). 
\end{align}
Now, let  $c_{k}^{m}$ be the Fourier coefficients of the trace of $v_{m}$.
By (\ref{eqYY}) we can see that 
\begin{align}\label{compik}
K'(\delta)&\geq ||v_{m}||^{2}_{\mathbb{X}^{s}_{m,T}}\geq |v_{m}(\cdot,0)|^{2}_{\mathbb{H}^{s}_{m,T}}\nonumber \\
&\geq \sum_{k\in \Z^{N}} \omega^{2s}|k|^{2s}|c_{k}^{m}|^{2},
\end{align}
which, together with (\ref{vti}), implies that 
\begin{align}\label{baroni}
|v_{m}(\cdot,0)|_{\mathbb{H}^{s}_{T}}\leq K''(\delta) \mbox{ for every } m\in (0,m_{0})
\end{align} 
that is $\textup{Tr}(v_{m})$ is bounded in $\mathbb{H}^{s}_{T}$.

Finally we estimate the $L^{2}_{loc}(\mathcal{S}_{T},y^{1-2s})$-norm of $v_{m}$ uniformly in $m$.

Fix $\alpha>0$ and let $v\in C^{\infty}_{T}(\overline{\R^{N+1}_{+}})$ such that $||v_{m}||_{\mathbb{X}^{s}_{m,T}}<\infty$.
For any $x\in [0,T]^{N}$ and $y\in [0, \alpha]$, we have
$$
v(x,y)=v(x,0)+\int_{0}^{y} \partial_{y} v(x,t) dt.
$$
By using $(a+b)^{2}\leq 2a^{2}+2b^{2}$ for all $a, b\geq 0$ we obtain
$$
|v(x,y)|^2 \leq 2  |v(x,0)|^{2}+2\Bigl(\int_{0}^{y}|\partial_{y} v(x,t)| dt\Bigr)^{2},
$$
and applying the H\"older inequality we deduce
\begin{equation}\label{vtii5}
|v(x,y)|^2 \leq 2 \Bigl[ |v(x,0)|^{2}+\Bigl(\int_{0}^{y} t^{1-2s}|\partial_{y} v(x,t)|^{2}dt\Bigr)\frac{y^{2s}}{2s}\,  \Bigr].
\end{equation}
Multiplying both members by $y^{1-2s}$ we have
\begin{equation}\label{vtii}
y^{1-2s}|v(x,y)|^2 \leq 2 \Bigl[ y^{1-2s}|v(x,0)|^{2}+\Bigl(\int_{0}^{y} t^{1-2s} |\partial_{y} v(x,t)|^{2}dt\Bigr)\frac{y}{2s} \Bigr].
\end{equation}
Integrating (\ref{vtii}) over $(0,T)^{N}\times (0,\alpha)$ we have
\begin{align}\label{nash}
||v||_{L^{2}((0,T)^{N}\times (0,\alpha),y^{1-2s})}^2 &\leq \frac{\alpha^{2-2s}}{1-s} |v(\cdot,0)|_{L^{2}(0,T)^{N}}^{2}+\frac{\alpha ^{2}}{2s} ||\partial_{y} v||_{L^{2}(\mathcal{S}_{T},y^{1-2s})}^{2}.
\end{align}
By density, the above inequality holds for any $v\in \mathbb{X}^{s}_{m,T}$.

Then, by using (\ref{nash}) and exploiting (\ref{vti}) and (\ref{nablavti}) we have 
\begin{align*}
||v_{m}||_{L^{2}((0,T)^{N}\times (0,\alpha),y^{1-2s})}^2 &\leq \frac{\alpha^{2-2s}}{1-s} |v_{m}(\cdot,0)|_{L^{2}(0,T)^{N}}^{2}+\frac{\alpha ^{2}}{2s} ||\partial_{y} v_{m}||_{L^{2}(\mathcal{S}_{T},y^{1-2s})}^{2} \\
& \leq C(\alpha,s) K(\delta)^{2}+C'(\alpha,s) K'(\delta)
\end{align*}
for any $0<m<m_{0}$.

As a consequence, we can extract a subsequence, that for simplicity we will denote again with  $(v_{m})$, and a function $v$ such that 
\begin{itemize}
\item $v\in L^{2}_{loc}(\mathcal{S}_{T},y^{1-2s})$ and $\nabla v\in L^{2}(\mathcal{S}_{T},y^{1-2s})$;
\item $v_{m}\rightharpoonup v$ in $L^{2}_{loc}(\mathcal{S}_{T},y^{1-2s})$ as $m\rightarrow 0$; 
\item $\nabla v_{m}\rightharpoonup \nabla v$ in $L^{2}(\mathcal{S}_{T},y^{1-2s})$ as $m\rightarrow 0$; 
\item $v_{m}(\cdot,0)\rightharpoonup v(\cdot,0)$  in $\mathbb{H}^{s}_{T}$ and $v_{m}(\cdot,0)\rightarrow v(\cdot,0)$ in $L^{q}(0,T)^{N}$ as $m\rightarrow 0$, for any $q\in [2, \frac{2N}{N-2s})$.
\end{itemize}

Now we prove that $v$ is a weak solution to 
\begin{equation}\label{Ros}
\left\{
\begin{array}{ll}
-\dive(y^{1-2s} \nabla v) =0 &\mbox{ in }\mathcal{S}_{T}:=(0,T)^{N} \times (0,\infty)  \\
v_{| {\{x_{i}=0\}}}= v_{| {\{x_{i}=T\}}} & \mbox{ on } \partial_{L}\mathcal{S}_{T}:=\partial (0,T)^{N} \times [0,\infty) \\
\frac{\partial v}{\partial \nu^{1-2s}}=f(x,v)   &\mbox{ on }\partial^{0}\mathcal{S}_{T}:=(0,T)^{N} \times \{0\}
\end{array}.
\right.
\end{equation}

Fix $\varphi \in \mathbb{X}^{s}_{T}$.
We know that $v_{m}$ satisfies 
\begin{equation}\label{spqr}
\iint_{\mathcal{S}_{T}} y^{1-2s}(\nabla v_{m} \nabla \eta+m^{2}v_{m}\eta) \; dxdy=\int_{\partial^{0} \mathcal{S}_{T}} [m^{2s}v_{m}+f(x,v_{m})]\eta  \; dx
\end{equation}
for every $\eta \in \mathbb{X}^{s}_{m,T}$.
Now, we consider $\xi\in \mathcal{C}^{\infty}([0,\infty))$ defined as follows
\begin{equation}\label{xidef}
\left\{
\begin{array}{cc}
\xi=1 &\mbox{ if } 0\leq y\leq 1 \\
0\leq \xi \leq 1 &\mbox{ if } 1\leq y\leq 2 \\ 
\xi=0 &\mbox{ if } y\geq 2 
\end{array}.
\right.
\end{equation}
We set $\xi_{R}(y)=\xi(\frac{y}{R})$ for $R>1$. Then choosing $\eta=\varphi \xi_{R}\in \mathbb{X}^{s}_{m,T}$ in (\ref{spqr}) and taking the limit as $m\rightarrow 0$ we have 
\begin{equation}
\iint_{\mathcal{S}_{T}} y^{1-2s}\nabla v \nabla (\varphi \xi_{R}) \; dxdy=\int_{\partial^{0} \mathcal{S}_{T}} f(x,v)\varphi \; dx.
\end{equation}
Taking the limit as $R\rightarrow \infty$ we deduce that $v$ verifies
$$
\iint_{\mathcal{S}_{T}} y^{1-2s} \nabla v \nabla \varphi \; dxdy-\int_{\partial^{0} \mathcal{S}_{T}} f(x,v)\varphi  \; dx=0 \quad \forall \varphi \in \mathbb{X}^{s}_{T}.
$$
Now we want to prove that $v\not \equiv 0$.
Let $\xi \in \mathcal{C}^{\infty}([0,\infty))$ as in (\ref{xidef}) and we note that $\xi v \in \mathbb{X}^{s}_{m,T}$. 

Then
\begin{align*}
0&= \langle \mathcal{J}'_{m}(v_{m}), \xi v \rangle \\
&= \iint_{{\mathcal{S}}_{T}} y^{1-2s}(\nabla v_{m} \nabla(\xi v) + m^{2} v_{m} \xi v) \, dxdy - m^{2s}\int_{\partial^{0} \mathcal{S}_{T}} v_{m} v \, dx \\
& - \int_{\partial^{0} \mathcal{S}_{T}} f(x,v_{m})v \,dx 
\end{align*}
and taking the limit as $m\rightarrow 0$ we get
\begin{equation}\label{6.8}
0=\iint_{\mathcal{S}_{T}} y^{1-2s} \nabla v \nabla(\xi v) \, dxdy - \int_{\partial^{0} \mathcal{S}_{T}} f(x,v)v \, dx.
\end{equation}
By using the facts (\ref{zero}), (\ref{alpham}), $\langle \mathcal{J}'_{m}(v_{m}), v_{m}\rangle =0$ and $F\geq 0$, we have
\begin{align}\label{mF}
2\lambda &\leq 2\mathcal{J}_{m}(v_{m}) + m^{2s} |v_{m}(\cdot,0)|_{L^{2}(0,T)^{N}}^{2} + 2\int_{\partial^{0} \mathcal{S}_{T}} F(x, v_{m}) \, dx \nonumber \\
&=\|v_{m}\|_{\mathbb{X}^{m}_{T}}^{2} =m^{2s} |v_{m}(\cdot,0)|_{L^{2}(0,T)^{N}}^{2} + \int_{\partial^{0} \mathcal{S}_{T}} f(x,v_{m}) v_{m} \, dx.
\end{align}
Taking the limit in $(\ref{mF})$ as $m\rightarrow 0$ we obtain 
\begin{equation}\label{6.9}
2\lambda \leq \int_{\partial^{0} \mathcal{S}_{T}} f(x, v)v \, dx.
\end{equation}
Hence, (\ref{6.8}) and (\ref{6.9}) give
\begin{equation*}
0<2\lambda \leq \int_{\partial^{0} \mathcal{S}_{T}} f(x, v)v \, dx = \iint_{\mathcal{S}_{T}} y^{1-2s}\nabla v \nabla(\xi v) \, dxdy.
\end{equation*}
that is $v$ is not a trivial solution to (\ref{Ros}). 

Finally, we show that $v\in \mathcal{C}^{0,\alpha}([0,T]^{N})$, for some $\alpha \in (0,1)$.
We start proving that $v(\cdot,0)\in L^{q}(0,T)^{N}$ for any $q<\infty$. 
We proceed as in the proof of Lemma \ref{lemmino} and we use the estimate (\ref{baroni}).   

Let $w_{m}=v_{m}v^{2\beta}_{m,K}$ where $v_{m,K}=\min\{|v_{m}|,K\}$, $K>1$ and $\beta\geq 0$. Then, replacing $vv_{K}^{2\beta}$ by $v_{m}v^{2\beta}_{m,K}$ in (\ref{S1}), we can see that
\begin{align}
||v_{m}v_{m,K}^{\beta}||_{\mathbb{X}^{m}_{T}}^{2}\leq c_{\beta} \int_{(0,T)^{N}} [m^{2s}v_{m}^{2}v_{m,K}^{2\beta} + f(x,v_{m})v_{m} v_{m,K}^{2\beta}] dx 
\end{align}
where $c_{\beta}=1+\frac{\beta}{2}\geq 1$.
Using Lemma \ref{lfF} with $\varepsilon=1$ we get
\begin{equation}
m^{2s}v_{m}^{2}v_{m,K}^{2\beta} + f(x,v_{m})v_{m} v_{m,K}^{2\beta}\leq (m^{2s}+2)v_{m}^{2}v_{m,K}^{2\beta}+(p+1)C_{1}|v_{m}|^{p-1}v^{2}v_{m,K}^{2\beta}.
\end{equation}
Since $v_{m}$ converges strongly in $L^{\frac{N(p-1)}{2s}}(0,T)^{N}$ (because of $\frac{N(p-1)}{2s}< 2^{\s}_{s}$), we can assume that, up to subsequences, there exists a function $z \in L^{\frac{N(p-1)}{2s}}(0,T)^{N}$ such that $|v_{m}(x,0)|\leq z(x)$ in $(0,T)^{N}$, for every $m<m_{0}$.

Therefore, there exist a constant $c=m^{2s}_{0}+2+(p+1)C_{1}$ and a function $h:=1+z^{p-1}\in L^{\frac{N}{2s}}(0,T)^{N}$, $h\geq 0$ and independent of $K$, $m$ and $\beta$ such that
\begin{equation}
m^{2s}v_{m}^{2}v_{m,K}^{2\beta} + f(x,v_{m})v_{m}v_{m,K}^{2\beta}\leq (c+h)v_{m}^{2}v_{m,K}^{2\beta} \mbox{ on } \partial^{0}\mathcal{S}_{T}.
\end{equation}
As a consequence
\begin{equation}
||v_{m}v_{m,K}^{\beta}||_{\mathbb{X}^{s}_{m,T}}^{2}\leq c_{\beta} \int_{(0,T)^{N}} (c+h)v_{m}^{2}v_{m,K}^{2\beta} dx.
\end{equation}
Taking the limit as $K\rightarrow \infty$ ($v_{m,K}$ is increasing with respect to $K$) we get
\begin{equation}\label{i177}
|||v_{m}|^{\beta+1}||_{\mathbb{X}^{s}_{m,T}}^{2}\leq cc_{\beta} \int_{(0,T)^{N}} |v_{m}|^{2(\beta +1)} + c_{\beta}\int_{(0,T)^{N}}  h|v_{m}|^{2(\beta +1)}dx.
\end{equation}
For any $M>0$, let $A_{1}=\{h\leq M\}$ and $A_{2}=\{h>M\}$.
Then
\begin{equation}\label{i277}
\int_{(0,T)^{N}}  h|v_{m}(\cdot,0)|^{2(\beta +1)}dx\leq M||v_{m}(\cdot,0)|^{\beta+1}|_{L^{2}(0,T)^{N}}^{2}+\varepsilon(M)||v_{m}(\cdot,0)|^{\beta+1}|_{L^{2^{\s}_{s}}(0,T)^{N}}^{2}
\end{equation}
where $\varepsilon(M)=\Bigl(\int_{A_{2}} h^{\frac{N}{2s}} dx \Bigr)^{\frac{2s}{N}}\rightarrow 0$ as $M\rightarrow \infty$.
Taking into account (\ref{i177}), (\ref{i277}), we have
\begin{equation}\label{regv77}
|||v_{m}|^{\beta+1}||_{\mathbb{X}^{s}_{m,T}}^{2}\leq c_{\beta}(c+M)||v_{m}(\cdot,0)|^{\beta+1}|_{L^{2}(0,T)^{N}}^{2}+c_{\beta}\varepsilon(M)||v_{m}(\cdot,0)|^{\beta+1}|_{L^{2^{\s}}(0,T)^{N}}^{2}.
\end{equation}
Now, by (\ref{benyi}), we know that for every $w\in \mathcal{C}^{\infty}_{T}(\R^{N})$ with mean zero, there exists $\mu_{0}:=C_{2^{\s}_{s}}>0$, such that
\begin{equation}\label{puffo}
|w|_{L^{2^{\s}_{s}}(0,T)^{N}}\leq \mu_{0} \Bigl(\sum_{|k|\neq 0} \omega^{2s}|k|^{2s} |b_{k}|^{2}\Bigr)^{1/2},
\end{equation}
where $b_{k}$ are the Fourier coefficients of $w$. 

Therefore, if $w\in \mathcal{C}^{\infty}_{T}(\R^{N})$ and $\overline{w}:=\frac{1}{T^{N}}\int_{(0,T)^{N}} w(x)dx$, by using H\"older inequality we can see that
\begin{align}
|w|_{L^{2^{\s}_{s}}(0,T)^{N}}&\leq |w-\overline{w}|_{L^{2^{\s}_{s}}(0,T)^{N}}+|\overline{w}|_{L^{2^{\s}_{s}}(0,T)^{N}} \nonumber \\
&\leq \mu_{0} \Bigl(\sum_{|k|\neq 0} \omega^{2s}|k|^{2s} |b_{k}|^{2}\Bigr)^{1/2}+|\overline{w}|_{L^{2^{\s}_{s}}(0,T)^{N}} \nonumber  \\
&\leq \mu_{0} \Bigl(\sum_{|k|\neq 0} \omega^{2s}|k|^{2s} |b_{k}|^{2}\Bigr)^{1/2}+\mu_{1} |w|^{2}_{L^{2}(0,T)^{N}} \nonumber \\
&\leq \mu_{0} |w|^{2}_{\mathbb{H}^{s}_{m,T}}+\mu_{1} |w|^{2}_{L^{2}(0,T)^{N}} \label{berti}
\end{align}
where $\mu_{1}=T^{\frac{N-2}{2}}>0$.

Taking into account  (\ref{regv77}), (\ref{berti}) and (\ref{eqYY}), we deduce that
\begin{align}
&||v_{m}(\cdot,0)|^{\beta+1}|^{2}_{L^{2^{\s}_{s}}(0,T)^{N}}-\mu_{1} ||v_{m}(\cdot,0)|^{\beta+1}|^{2}_{L^{2}(0,T)^{N}} \\
& \leq \mu_{0} ||v_{m}(\cdot,0)|^{\beta+1}|^{2}_{\mathbb{H}^{s}_{m,T}} \nonumber \\
&\leq \mu_{0}  |||v_{m}|^{\beta+1}||_{\mathbb{X}^{s}_{m,T}}^{2} \nonumber \\
& \leq \mu_{0} \Bigl[ c_{\beta}(c+M)||v_{m}(\cdot,0)|^{\beta+1}|_{L^{2}(0,T)^{N}}^{2} \nonumber \\
&+c_{\beta}\varepsilon(M)||v_{m}(\cdot,0)|^{\beta+1}|_{L^{2^{\s}_{s}}(0,T)^{N}}^{2} \Bigr]  \label{iterbeta}.
\end{align}

Choosing $M$ large so that $c_{\beta} \mu_{0} \varepsilon(M)<\frac{1}{2}$, by (\ref{iterbeta}) we obtain
\begin{equation}\label{iter77}
||v_{m}(\cdot,0)|^{\beta+1}|_{L^{2^{\s}_{s}}(0,T)^{N}}^{2} \leq 2[\mu_{0}c_{\beta}(c+M)+\mu_{1}] ||v_{m}(\cdot,0)|^{\beta+1}|^{2}_{L^{2}(0,T)^{N}}.
\end{equation}
Let us notice that, by (\ref{baroni}) and $\mathbb{H}^{s}_{T}\subset L^{2^{\s}_{s}}(0,T)^{N}$, we get
\begin{align}\label{bessi}
|v_{m}(\cdot,0)|_{L^{2^{\s}_{s}}(0,T)^{N}} \leq K'''(\delta),
\end{align}
for any $m<m_{0}$.
By applying (\ref{iter77}) with $\beta+1=\frac{N}{N-2s}$ (that is $\beta=\frac{2s}{N-2s}$)  and by using (\ref{bessi}) we have that
$$
||v_{m}|^{\frac{N}{N-2s}}|_{L^{2^{\s}_{s}}(0,T)^{N}}^{2} \leq 2[c_{\frac{2s}{N-2s}} \mu_{0} (c+M)+\mu_{1}] K'''(\delta)^{\frac{2N}{N-2s}},
$$
and taking the limit as $m \rightarrow 0$, we deduce $v(\cdot,0)\in L^{\frac{2N^{2}}{(N-2s)^{2}}}(0,T)^{N}$.

By using the formula (\ref{iter77}), we find, after $k$ iterations, that $v(\cdot,0)\in L^{\frac{2N^{k}}{(N-2s)^{k}}}(0,T)^{N}$ for all $k \in \mathbb{N}$. Then $v(\cdot,0)\in L^{q}(0,T)^{N}$ for all $q\in[2,\infty)$, and by invoking Proposition $3.5.$ in \cite{FallFelli}, we can conclude that $v\in \mathcal{C}^{0,\alpha}([0,T]^{N})$, for some $\alpha \in (0,1)$.

\end{document}